\documentclass[12pt,a4paper]{article}
\usepackage{oldgerm}
\usepackage{amsmath}
\usepackage{amssymb}
\usepackage{amsthm}
\newtheorem{thm}{Theorem}[section]
\newtheorem{prop}[thm]{Proposition}
\newtheorem{lem}[thm]{Lemma}
\newtheorem{cor}[thm]{Corollary}
\newtheorem{definition}[thm]{Definition}
\newtheorem{rem}[thm]{Remark}
\newcommand{\hK}{hyper-K\"ahler\ }
\newcommand{\HK}{Hyper-K\"ahler\ }

\newcommand{\A}{A_\infty}

\newcommand{\Z}{\mathbb{Z}}
\newcommand{\N}{\mathbb{N}}
\newcommand{\R}{\mathbb{R}}
\newcommand{\C}{\mathbb{C}}

\begin{document}
\title{The nonuniqueness of the tangent cones at infinity of 
Ricci-flat manifolds}
\author{Kota Hattori}
\date{}
\maketitle
{\begin{center}
{\it Keio University\\
3-14-1 Hiyoshi, Kohoku, Yokohama 223-8522, Japan\\
hattori@math.keio.ac.jp}
\end{center}}
\maketitle
{\abstract It is shown by Colding and Minicozzi 
the uniqueness of the tangent cones at infinity of Ricci-flat manifolds 
with Euclidean volume growth which has at least one 
tangent cone at infinity with a smooth cross section. 
In this paper we raise an example of the Ricci-flat manifold 
implying that the assumption for the volume growth in the above result 
is essential. 
More precisely, we construct a complete Ricci-flat manifold 
of dimension $4$ with non-Euclidean volume growth 
who has infinitely many tangent cones at infinity 
and one of them has a smooth cross section.}

\section{Introduction}
For a complete Riemannian manifold $(X,g)$ with nonnegative 
Ricci curvature, 
it is shown by Gromov's Compactness Theorem that if one 
take a sequence 
\begin{align*}
a_1>a_2>\cdots >a_i >\cdots >0
\end{align*}
such that $\lim_{i\to \infty}a_i = 0$, then there is a subsequence 
$\{ a_{i(j)}\}_j$ such that 
$(X,a_{i(j)} g, p)$ converges to a pointed metric space $(Y,d,q)$ as $j\to \infty$ 
in the sense of the pointed Gromov-Hausdorff topology 
\cite{gromov1981}\cite{gromov2007}. 
The limit $(Y,d,q)$ is called the tangent cone at infinity of $(X,g)$.
In general, the pointed Gromov-Hausdorff limit might 
depend on the choice of $\{ a_i\}_i$ or its subsequences. 

The tangent cone at infinity is said to be unique if the isometry classes of 
the limits are independent of the choice of $\{ a_i\}$ and its subsequences, 
and Colding and Minicozzi showed the next uniqueness theorem 
under the certain assumptions. 
\begin{thm}[\cite{Colding-Minicozzi2014}]
Let $(X,g)$ be a Ricci-flat manifold with Euclidean volume growth, 
and suppose that one of the tangent cone at infinity has a smooth cross section. 
Then the tangent cone at infinity of $(X,g)$ is unique. 
\label{thm colding minicozzi}
\end{thm}
Among the assumptions in Theorem \ref{thm colding minicozzi}, 
the Ricci-flat condition is essential since there are several examples 
of complete Riemannian manifolds with nonnegative Ricci curvature 
and Euclidean volume growth, of whom one of the tangent cones at infinity 
has smooth cross section, but the tangent cones at infinity is not unique 
\cite{perelman1997}\cite{Colding-Naber2013}. 

Here, let $\mathcal{T}(X,g)$ be the set of all of the isometry classes 
of the tangent cones at infinity of $(X,g)$. 
In this paper, the isometry between pointed metric spaces means 
the bijective map preserving the metrics and the base points. 
It is known that $\mathcal{T}(X,g)$ is closed with respect to 
the pointed Gromov-Hausdorff topology, and has the natural 
continuous $\R^+$-action defined by the rescaling of metrics. 
The uniqueness of the tangent cones at infinity means that 
$\mathcal{T}(X,g)$ consists of only one point.

In this paper, we show that the assumption for the volume growth 
in Theorem \ref{thm colding minicozzi} is essential. 
More precisely, we obtain the next main result. 
\begin{thm}
There is a complete Ricci-flat manifold $(X,g)$ of dimension $4$ 
such that $\mathcal{T}(X,g)$ is homeomorphic to $S^1$. 
Moreover, $\R^+$-action on $\mathcal{T}(X,g)$ fixes 
$(\R^3,d_0^\infty,0)$, $(\R^3,h_0,0)$, $(\R^3,h_1,0)$, 
where $h_0=\sum_{i=1}^3 (d\zeta_i)^2$ is the Euclidean metric, 
$h_1=\frac{1}{|\zeta|}h_0$, and $d_0^\infty$ is the completion of 
the Riemannian metric 
\begin{align*}
\int_0^\infty \frac{dx}{|\zeta-(x^\alpha,0,0)|}\cdot h_0,
\end{align*}
and $\R^+$ acts freely on
\begin{align*}
\mathcal{T}(X,g)\backslash \{ (\R^3,d_0^\infty,0),
\ (\R^3,h_0,0),\ (\R^3,h_1,0)\}. 
\end{align*}
Here, $\zeta=(\zeta_1,\zeta_2,\zeta_3)$ is 
the Cartesian coordinate on $\R^3$.
\label{main result}
\end{thm}
Here, we mention more about the metric spaces appearing 
in Theorem \ref{main result}. 
For $0\le S<T\le\infty$, denote by $d_S^T$ the metric on $\R^3$ 
induced by the Riemannian metric 
\begin{align*}
\int_S^T \frac{dx}{|\zeta-(x^\alpha,0,0)|}\cdot h_0. 
\end{align*}
For $(X,g)$ in Theorem \ref{main result}, we show that 
$\mathcal{T}(X,g)$ contains $\{ (\R^3,d_0^T,0); T\in\R^+ \}$, 
$\{ (\R^3,d_S^\infty,0); S\in\R^+ \}$ and 
$\{ (\R^3,h_0 + \theta h_1,0); \theta\in\R^+\}$. 
Here, we can check easily that $d_0^T$ and $d_S^\infty$ 
are homothetic to $d_0^1$ and $d_1^\infty$, respectively. 
We can show that 
\begin{align*}
(\R^3,d_0^T,0) &\xrightarrow[T\to \infty]{GH} (\R^3,d_0^\infty,0), \quad 
(\R^3,d_0^T,0) \xrightarrow[T\to 0]{GH} (\R^3,h_1,0), 
\\
(\R^3,d_S^\infty,0) &\xrightarrow[S\to \infty]{GH} (\R^3,h_0,0), \quad 
(\R^3,d_S^\infty,0) \xrightarrow[S\to 0]{GH} (\R^3,d_0^\infty,0), 
\\
(\R^3,h_0 + \theta h_1,0) &\xrightarrow[\theta\to \infty]{GH} (\R^3,h_1,0), \quad 
(\R^3,h_0 + \theta h_1,0) \xrightarrow[\theta\to 0]{GH} (\R^3,h_0,0).
\end{align*}
Both of $(\R^3,h_0)$ and $(\R^3,h_1)$ 
can be regarded as the Riemannian cones with respect to the dilation 
$\zeta\mapsto \lambda\zeta$ on $\R^3$. 
Although the dilation also pulls back 
$d_0^\infty$ to $\lambda^\frac{\alpha + 1}{2\alpha}d_0^\infty$, 
$(\R^3,d_0^\infty)$ does not become the metric cone 
with respect to this dilation 
since $\mathbf{l} = \{ (t,0,0)\in\R^3;\ t\ge 0\}$ is not a ray. 
In fact, any open intervals contained in $\mathbf{l}$ 
have infinite length with respect to $d_0^\infty$.

In general, tangent cones at infinity of 
complete Riemannian manifolds with nonnegative Ricci curvature 
and Euclidean volume growth are metric cones \cite{Cheeger-Colding1996}. 
In our case, it is shown in Section \ref{geometry} that 
$(\R^3,d_0^\infty,0)$ never become the metric cone of 
any metric space.

The Ricci-flat manifold $(X,g)$ appeared in Theorem \ref{main result} 
is one of the \hK manifolds of type $\A$, 
constructed by Anderson, Kronheimer and LeBrun in \cite{AKL1989} 
applying Gibbons-Hawking ansatz, 
and by Goto in \cite{goto1994} as 
\hK quotients. 
Combining Theorems \ref{thm colding minicozzi} and \ref{main result}, 
we can see that the volume growth of $(X,g)$ should not be Euclidean. 
In fact, the author has computed the volume growth of 
the \hK manifolds of type $\A$ in \cite{hattori2011}, 
and showed that they are always 
greater than cubic growth and less than Euclidean growth. 
To construct $(X,g)$, we ``mix'' the \hK manifold of type $\A$ 
whose volume growth is $r^a$ for some $3<a<4$, and 
$\R^4$ equipped with the standard \hK structure. 
Unfortunately, the author could not compute the volume growth of 
$(X,g)$ in Theorem \ref{main result} explicitly.

In this paper, we can show that a lot of metric spaces may arise 
as the Gromov-Hausdorff limit of \hK manifolds of type $\A$. 
Let 
\begin{align*}
I\in \mathcal{B}_+(\R^+) 
:= \{ J\subset\R^+;\ J\ \mbox{{\rm is a Borel set 
of nonzero Lebesgue measure}}\}
\end{align*}
and denote by $d_I$ the metric on 
$\R^3$ induced by the Riemannian metric 
$\int_I \frac{dx}{|\zeta-(x^\alpha,0,0)|}\cdot h_0$. 
Then we have the following result. 

\begin{thm}\label{main result2}
There is a complete Ricci-flat manifold $(X,g)$ of dimension $4$ 
such that $\mathcal{T}(X,g)$ contains 
\begin{align*}
\{ (\R^3,d_I,0);\ I\in \mathcal{B}_+(\R^+)\} / \mbox{{\rm isometry}}.
\end{align*}
\end{thm}
Since $d_S^\infty$ and $d_0^T$ are contained in $\mathcal{T}(X,g)$ 
in the above theorem, then their limits $h_0$ and $\frac{1}{|\zeta|}h_0$ 
are also contained in $\mathcal{T}(X,g)$. 
The author does not know whether any other metric spaces 
are contained in $\mathcal{T}(X,g)$.

Theorems \ref{main result} and \ref{main result2} are 
shown along the following process.
The above-mentioned \hK manifolds are constructed from 
infinitely countable subsets $\Lambda$ in $\R^3$ such that 
$\sum_{\lambda\in\Lambda}\frac{1}{1+|\lambda|} <\infty$.
We denote it by $(X,g_\Lambda)$ and fix the base point $p\in X$. 
From the construction, $(X,g_\Lambda)$ has a natural $S^1$-action 
preserving $g_\Lambda$ and the \hK structure, 
then we obtain a \hK moment map $\mu_\Lambda:X\to \R^3$ 
such that $\mu_\Lambda(p)=0$, 
which is a surjective map whose generic fibers are $S^1$. 
There is a unique distance function $d_\Lambda$ on $\R^3$ such that 
$\mu_\Lambda$ is a submetry. 
Here, submetries are the generalization of Riemannian submersions 
to the category of metric spaces. 
For $a>0$ we can see 
$a g_\Lambda = g_{a \Lambda}$, hence 
by taking $a_i>0$ such that $\lim_{i\to \infty} a_i = 0$, 
we obtain a sequence of submetries $\mu_{a_i \Lambda}:X\to \R^3$. 
Now, assume that 
$\{ (\R^3,d_{a_i \Lambda},0)\}_i$ converges to a metric space 
$(\R^3,d_\infty,0)$ for some $d_\infty$ 
in the pointed Gromov-Hausdorff topology, 
and the diameters of fibers of $\mu_{a_i \Lambda}$ converges to $0$ 
in some sense.
Then we can show $(\R^3,d_\infty,0)$ is the Gromov-Hausdorff limit of 
$\{ (X,g_{a_i \Lambda},p)\}_i$. 
We raise a concrete example of $\Lambda$ and sequences $\{ a_i\}_i$, 
then obtain several limit spaces. 
Among them, it is shown in Section \ref{geometry} that 
$(\R^3,d_0^\infty)$ is not a polar space in the sense of Cheeger and Colding 
\cite{Cheeger-Colding1997}.

This paper is organized as follows. 
We review the construction of \hK manifolds of type $\A$ 
and \hK moment map $\mu_\Lambda$ in Section \ref{secAKL}.
Then we review the notion of submetry in 
Section \ref{secsubmetry}, and 
the notion of Gromov-Hausdorff topology in Section \ref{secGH}. 
In Section \ref{sec tan}, we construct a submetry 
$\mu_a$ from $(X,g_{a \Lambda})$ to 
$(\R^3,d_a)$ by using $\mu_\Lambda$ and dilation, 
where $d_a$ is the metric induced by the Riemannian metric 
$\Phi_a(\zeta)h_0$. Here, $\Phi_a$ is a positive valued 
harmonic function determined by $\Lambda$ and some constants. 
Then we see that the convergence of $\{ (X,g_{a_i \Lambda})\}_i$ 
can be reduced to the convergence of 
$\{ (\R^3,d_{a_i})\}_i$. 
In Sections \ref{sec const} and \ref{sec dis}, 
we raise concrete examples of $\Lambda$ 
and fix $a>0$, 
then estimate the difference of $\Phi_a$ and 
another positive valued harmonic function $\Phi_\infty$, 
which induces the metric $d_\infty$ on $\R^3$. 
In Section \ref{sec conv}, we observe some examples 
by applying the results in Sections 
\ref{sec const} and \ref{sec dis}, then show 
Theorems \ref{main result} and \ref{main result2}. 
In Section \ref{geometry}, 
we prove that $(\R^3,d_0^\infty)$ is not a polar space. 

\paragraph{Acknowledgment.}
The author would like to thank Professor Shouhei Honda 
who invited the author to this attractive topic, and 
also thank him for giving the several advice on this paper. 
The author also would like to thank the referee for careful reading 
and several useful comments. 
Thanks to his pointing out, 
the author could make the main results much stronger. 
The author was supported by 
Grant-in-Aid for Young Scientists (B) Grant Number 
16K17598.

\section{\HK manifolds of type $\A$}\label{secAKL}
Here we review shortly the construction of \hK manifolds of 
type $\A$, along \cite{AKL1989}.

Let $\Lambda\subset \R^3$ be a countably infinite subset 
satisfying the convergence condition
\begin{align*}
\sum_{\lambda\in\Lambda} \frac{1}{1+|\lambda |} < \infty,
\end{align*}
and take a positive valued harmonic function $\Phi_\Lambda$ over 
$\R^3\backslash \Lambda$ defined by 
\begin{align*}
\Phi_\Lambda(\zeta) := \sum_{\lambda\in\Lambda} \frac{1}{|\zeta - \lambda |}.
\end{align*}
Then $*d\Phi_\Lambda \in \Omega^2(\R^3\backslash \Lambda)$ is a closed $2$-form 
where $*$ is the Hodge's star operator of the Euclidean metric, 
and we have an integrable cohomology class $[\frac{1}{4\pi}*d\Phi_\Lambda]\in H^2(\R^3\backslash \Lambda,\Z)$, 
which is equal to the $1$st Chern class of a principal $S^1$-bundle $\mu=\mu_\Lambda: X^*\to \R^3\backslash \Lambda$.
For every $\lambda\in \Lambda$, we can take a sufficiently small open ball $B\subset \R^3$ centered at $\lambda$ 
which does not contain any other elements in $\Lambda$.
Then $\mu : \mu^{-1}(B\backslash \{ \lambda \}) \to B\backslash \{ \lambda \}$ is isomorphic to 
Hopf fibration $\mu_0: \R^4\backslash \{ 0\} \to \R^3\backslash \{ 0\}$ as principal $S^1$-bundles, 
hence there exists a $C^\infty$ $4$-manifold $X$ and an open embedding $X^*\subset X$,  
and $\mu$ can be extended to an $S^1$-fibration 
\begin{align*}
\mu = (\mu_1,\mu_2,\mu_3): X\to \R^3.
\end{align*}
Moreover we may write $X\backslash X^* = \{ p_\lambda;\ \lambda\in\Lambda\}$ 
and $\mu (p_\lambda) = \lambda$.
Next we take an $S^1$-connection $\Gamma\in \Omega^1(X^*)$ on 
$X^*\to \R^3\backslash \Lambda$, whose curvature form is given by $*d\Phi_\Lambda$.
Then $\Gamma$ is uniquely determined up to exact $1$-form on $\R^3\backslash \Lambda$.
Now, we obtain a Riemannian metric 
\begin{align*}
g_\Lambda := (\mu^*\Phi_\Lambda)^{-1} \Gamma^2 + \mu^*\Phi_\Lambda\sum_{i=1}^3 (d\mu_i)^2
\end{align*}
on $X^*$, which can be extended to a smooth Riemannian metric $g_\Lambda$ over $X$ by 
taking $\Gamma$ appropriately.
\begin{thm}[\cite{AKL1989}]
Let $(X,g_\Lambda)$ be as above. 
Then it is a complete \hK (hence Ricci-flat) metric of dimension $4$.
\label{AKLmetric}
\end{thm}

Since $S^1$ acts on $(X,g_\Lambda)$ isometrically, it is easy to check that
\begin{align*}
\mu :(X^*,g_\Lambda) \to (\R^3\backslash \Lambda,\Phi_\Lambda\cdot h_0)
\end{align*}
is a Riemannian submersion, 
where $h_0$ is the Euclidean metric on $\R^3$.

Next we consider the rescaling of $(X,g_\Lambda)$.
For $a>0$, put $a\Lambda := \{ a\lambda \in \R^3;\lambda\in\Lambda\}$. 
Then it is easy to see
\begin{align*}
\Phi_{a\Lambda}(\zeta) = \sum_{\lambda\in\Lambda}\frac{1}{|\zeta - a\lambda|}
= a^{-1}\sum_{\lambda\in\Lambda}\frac{1}{|a^{-1}\zeta - \lambda|}
= a^{-1}\Phi_{\Lambda}(a^{-1}\zeta)
\end{align*}
and $\mu_{a\Lambda} = a\mu_\Lambda$, hence 
$\mu_{a\Lambda}^*\Phi_{a\Lambda} = a^{-1}\mu_\Lambda^*\Phi_\Lambda$ holds.
Thus we have
\begin{align*}
g_{a\Lambda} &= (\mu_{a\Lambda}^*\Phi_{a\Lambda})^{-1} \Gamma^2 + \mu_{a\Lambda}^*\Phi_{a\Lambda}\sum_{i=1}^3 
(d\mu_{a\Lambda,i})^2\\
&= a(\mu_{\Lambda}^*\Phi_{\Lambda})^{-1} \Gamma^2 + a \mu_{\Lambda}^*\Phi_{\Lambda}\sum_{i=1}^3 
(d\mu_{\Lambda,i})^2 = a g_\Lambda.
\end{align*}

\section{Submetry}\label{secsubmetry}
Throughout of this paper, the distance between $x$ and $y$ 
in a metric space $(X,d)$ is denoted by $d(x,y)$. 
If it is clear which metric is used,
we often write $|xy| = d(x,y)$

The map $\mu: X\to \R^3$ appeared in the previous section 
is not a Riemannian submersion, 
since $d\mu$ degenerates on $X\backslash X^*$ and 
$\Phi_\Lambda\cdot h_0$ does not defined on the whole of $\R^3$. 
However we can regard $\mu$ as a submetry, 
which is a notion introduced in \cite{berestovskii1987submetry}.

\begin{definition}[\cite{berestovskii1987submetry}]\normalfont
Let $X,Y$ be metric spaces, and $\mu : X\to Y$ be a map, 
which is not necessarily to be continuous.
Then $\mu$ is said to be a {\it submetry} if 
$\mu( D(p,r)) = D(\mu(p),r)$ holds for every 
$p\in X$ and $r>0$, 
where $D(p,r)$ is the closed ball of radius $r$ centered at $p$.
\end{definition}

Any proper Riemannian submersions between smooth Riemannian manifolds 
are known to be submetries.
Conversely, a submetry between smooth complete Riemannian manifolds 
becomes a $C^{1,1}$ Riemannian submersion \cite{berestovskii2000submetry}.

Now we go back to the setting in Section \ref{secAKL}.
Denote by $d_{\Lambda}$ the metric on $\R^3$ 
defined as the completion of the Riemannian distance 
induced from $\Phi_\Lambda\cdot h_0$. 
Since 
$\mu :(X^*,g_\Lambda) \to (\R^3\backslash \Lambda,\Phi_\Lambda\cdot h_0)$
is a Riemannian submersion, 
we have the following proposition.

\begin{prop}
Let $(X,g_\Lambda)$ be a \hK manifolds of type $\A$. 
The map $\mu: (X,d_{g_\Lambda}) \to (\R^3,d_\Lambda)$ is a submetry, 
where $d_{g_\Lambda}$ is the Riemannian distance induced from $g_\Lambda$. 
Moreover, we have 
\begin{align*}
d_\Lambda (q_0,q_1) = \inf_{p_1\in \mu^{-1}(q_1)} d_{g_\Lambda} (p_0,p_1) 
\end{align*}
for any $p_0\in \mu^{-1}(q_0)$
\end{prop}

\section{The Gromov-Hausdorff convergence}\label{secGH}
In this section, we discuss with the pointed Gromov-Hausdorff convergence 
of a sequence of pointed metric spaces equipped with submetries.
First of all, we review the definition of the pointed Gromov-Hausdorff convergence 
of pointed metric spaces.
Denote by $B(p,r) = B_X(p,r)$ the open ball of radius $r$ centered at $p$ in a metric space $X$.
\begin{definition}\normalfont
Let $(X,p)$ and $(X',p')$ be pointed metric spaces, 
and $r,\varepsilon$ be positive real numbers.
$f:B(p,r) \to X'$ is said to be an {\it $(r,\varepsilon)$-isometry} from $(X,p)$ to $(X',p')$
if $(1)$ $f(p) = p'$, 
$(2)$ $||xy| - |f(x)f(y)|| < \varepsilon$ holds for any $x,y\in B(p,r)$, 
$(3)$ $B(f(B(p,r)),\varepsilon)$ contains $B(p',r-\varepsilon)$.
\end{definition}
\begin{definition}\normalfont
Let $\{(X_i,p_i)\}_i$ be a sequence of pointed metric spaces.
Then $\{(X_i,p_i)\}_i$ is said to 
{\it converge to a metric space $(X,p)$ in the pointed 
Gromov-Hausdorff topology}, 
or $\{(X_i,p_i)\}_i\xrightarrow{\text{GH}}(X,p)$, 
if for any $r,\varepsilon>0$ there exists 
an positive integer $N_{(r,\varepsilon)}$ 
such that $(r,\varepsilon)$-isometry from $(X_i,p_i)$ to $(X,p)$ exists for 
every $l\ge N_{(r,\varepsilon)}$.
\end{definition}

For metric spaces $X,Y$, $q\in Y$ and a map $\mu:X\to Y$, 
define $\delta_{q,\mu}(r) \in \R_{\ge 0}\cup \{ \infty \}$ by
\begin{align*}
\delta_{q,\mu}(r) &:= \sup_{y\in B(q,r)} {\rm diam}(\mu^{-1}(y))\\
&= \sup_{\underset{x,x'\in \mu^{-1}(y)}{y\in B(q,r)}} |xx'|.
\end{align*}

\begin{prop}
Let $\{ (X,p)\}$ and $\{ (Y,q)\}$ be pointed metric spaces equipped with 
submetries $\mu:X\to Y$ satisfying $\mu(p)=q$, 
and $(Y_\infty,q_\infty)$ be another pointed metric space.
Assume that $\delta_{q,\mu}(r)<\infty$ and 
we have an $(r,\delta)$-isometry from $(Y,q)$ to $(Y_\infty,q_\infty)$. 
Then there exists an $(r,\delta + \delta_{q,\mu})$-isometry from 
$\{ (X,p)\}$ to $(Y_\infty,q_\infty)$.
\label{conv of submetry}
\end{prop}
\begin{proof}
Now, there is an $(r,\delta)$-isometry $f$ 
from $(Y,q)$ to $(Y_\infty,q_\infty)$.
Then it is easy to check that the composition 
$\hat{f}:= f\circ\mu$ is an 
$(r,\delta + \delta_{q,\mu})$-isometry from $(X,p)$ to $(Y_\infty,q_\infty)$.
\end{proof}

\section{Tangent cones at infinity}\label{sec tan}
Let $(X,d)$ be a metric space and $\{ a_i\}_i$ be a decreasing sequence of positive numbers 
converging to $0$.
If $(Y,q)$ is the pointed Gromov-Hausdorff limit of $\{ (X, a_id, p)\}_i$, then it is called 
an tangent cone at infinity of $X$.
It is clear that the limit does not depend on $p\in X$, 
but may depend on the choice of a sequence $\{ a_i\}_i$.

In this paper we are considering the tangent cones at infinity 
of $(X,d_{g_\Lambda})$.
In Section \ref{secAKL} we have seen $\sqrt{a}d_{g_\Lambda} = d_{g_{a\Lambda}}$ 
for $a>0$,
hence $\mu_{a\Lambda}: (X,\sqrt{a}d_{g_\Lambda}) \to (\R^3,d_{a\Lambda})$ 
is a submetry. 
By taking $N \in \R^+$ and 
the dilation 
$I_{N}:\R^3 \to\R^3$ defined by $I_{N}(\zeta) := \frac{1}{N}\zeta$,  
we have another submetry 
\begin{align*}
\mu_a:=I_{N}^{-1} \circ \mu_{a\Lambda}: (X,\sqrt{a}d_{g_\Lambda}) \to (\R^3, d_a:={I_{N}}^*d_{a\Lambda}).
\end{align*}
Here, ${I_{N}}^*d_{a\Lambda}$ is the completion of the Riemannian distance of 
\begin{align*}
{I_{N}}^*(\Phi_{a\Lambda}\cdot h_0) = {I_{N}}^*\Phi_{a\Lambda}\cdot \frac{1}{N^2} h_0
=  N\Phi_{N a\Lambda}\cdot \frac{1}{N^2} h_0 
=  \frac{1}{N}\Phi_{N a\Lambda}\cdot h_0,
\end{align*}
therefore we obtain $d_a$ which is the completion of 
the Riemannian metric $\Phi_a\cdot h_0$, 
where 
\begin{align*}
\Phi_a :=  \frac{1}{N}\Phi_{N a\Lambda}.
\end{align*}
In other words, $d_a$ is given by 
\begin{align}
d_a(x,y) = \inf_{\gamma \in {\rm Path} (x,y)} l_a(\gamma),\label{distance}
\end{align}
where ${\rm Path} (x,y)$ is the set of smooth paths in $\R^3$ joining 
$x,y\in \R^3$, 
and 
\begin{align}
l_a(\gamma) = \int_{t_0}^{t_1}\sqrt{\Phi_a(\gamma(t))} |\gamma'(t)| dt.\label{def length}
\end{align}

By the definition of $g_\Lambda$, one can see that 
the diameter of the fiber $\mu_{\Lambda}^{-1}(\zeta)$ is given by 
$\frac{\pi}{\sqrt{\Phi_{\Lambda}(\zeta)}}$. 
Accordingly, the diameter of $\mu_a^{-1}(\zeta)$ is given by
$\frac{\pi}{N\sqrt{\Phi_a(\zeta)}}$. 

For a metric $d_\infty$ on $\R^3$ and 
constants $r,\delta,\delta'>0$, 
we introduce the next assumptions.

\paragraph{(A1)}
The identity map 
\begin{align*}
{\rm id}_{\R^3} : (\R^3,d_a,0) \to (\R^3,d_\infty,0)
\end{align*}
is an $(r,\delta)$-isometry.

\paragraph{(A2)}
$\sup_{\zeta\in B_{d_a}(0,r)}\frac{\pi}{N\sqrt{\Phi_a(\zeta)}} < \delta'$
holds.

\paragraph{}
Then we obtain the next proposition by Proposition \ref{conv of submetry}.

\begin{prop}
Let $(X,g_\Lambda)$ and $\mu_a$ be as above, 
$p\in X$ satisfy $\mu_\Lambda(p) = 0$ 
and $d_\infty$ be a metric on $\R^3$. 
If {\bf (A1-2)} are satisfied for given constants $r,\delta,\delta'>0$, then 
$\mu_a$ is an $(r,\delta+\delta')$-isometry 
from $(X,ag_\Lambda,p)$ to $(\R^3,d_\infty,0)$.
\label{strategy}
\end{prop}

\section{Construction}\label{sec const}
Fix $\alpha>1$, and let 
\begin{align*}
\Lambda^{\alpha} &:= \{ (k^\alpha, 0, 0);\ k\in \Z_{\ge 0}\}.
\end{align*}
Take an increasing sequence of integers $0<K_0<K_1<K_2 <\cdots$.  

In this paper many constants will appear, 
and they may depend on $\alpha$ or $\{ K_n \}$.
However, we do not mind the dependence on these parameters.

Put 
\begin{align*}
\Lambda_{2n} &:= \{ (k^\alpha,0,0) \in\Lambda^{\alpha} ;\ K_{2n} \le k < K_{2n+1} \},\\
\Lambda &:= \bigcup_{n=0}^\infty \Lambda_{2n}.
\end{align*}

Since $\Lambda \subset \Lambda^{\alpha}$, 
we can see that $\sum_{\lambda\in\Lambda}\frac{1}{1+|\lambda|} < \infty$, 
accordingly we obtain a \hK manifold $(X,g_{\Lambda})$.

From now on we fix $a > 0$, $n\in \N$ and $P>0$, 
then put $N:=a^{\frac{-1}{1+\alpha}}P^\frac{1}{1+\alpha}$ 
and 
\begin{align*}
\Phi_a (\zeta) := \frac{1}{N}\Phi_{N a \Lambda}( \zeta ) 
= \sum_{ \lambda\in\Lambda } \frac{1}{N|\zeta - P N^{-\alpha}\lambda |}.
\end{align*}

Let $\mathbf{l}:=\{ (t,0,0)\in\R^3;\ t\ge 0 \}$ and put
\begin{align*}
K(R,D):=\{ \zeta\in\R^3;\ |\zeta|\le R,\ \inf_{y\in \mathbf{l}} |\zeta - y|\ge D\}.
\end{align*}
Here, $\inf_{y\in \mathbf{l}} |\zeta - y|$ is given by 
\[ \inf_{y\in \mathbf{l}} |\zeta - y| = 
\left \{
\begin{array}{ll}
\sqrt{|\zeta_\C|^2} & ({\rm if}\ \zeta_\R\ge 0)\\
|\zeta| & ({\rm if}\ \zeta_1< 0)
\end{array}
\right.
\]
for $\zeta = (\zeta_\R, \zeta_\C)\in \R^3=\R\oplus\C$.
For $0\le S< T\le \infty$, define a positive valued function $\Phi_{S,P}^T:\R^3\backslash \mathbf{l}:\to \R$ by
\begin{align*}
\Phi_{S,P}^T (\zeta):= \int_S^T \frac{dx}{|\zeta - P(x^\alpha, 0, 0)|}.
\end{align*}

Throughout this section, we put 
\begin{align*} 
S_n &:= \frac{K_{2n}}{N} = a^{\frac{1}{1+\alpha}}P^\frac{-1}{1+\alpha}K_{2n},\quad \quad \quad
T_n :=\frac{K_{2n+1}}{N} = a^{\frac{1}{1+\alpha}}P^\frac{-1}{1+\alpha}K_{2n+1}.
\end{align*}

\begin{prop}
We have 
\begin{align*}
\Big| \Phi_a ( \zeta ) 
- \sum_{n=0}^\infty \Phi_{S_n,P}^{T_n} (\zeta)\Big| 
\le \frac{2}{ND} = \frac{2}{D}\Big(\frac{a}{P}\Big)^\frac{1}{1+\alpha}
\end{align*}
for any $\zeta\in K(R,D)$.
\label{conv1}
\end{prop}
\begin{proof}
Since 
\begin{align*}
\Lambda_{2n} = \{ (k^\alpha,0,0);\ K_{2n} \le k < K_{2n+1}\},
\end{align*}
we have 
\begin{align*}
\sum_{ \lambda\in\Lambda_{2n} } \frac{1}{N|\zeta - PN^{-\alpha}\lambda |} 
&= \sum_{ k= K_{2n} }^{ K_{2n+1}-1 } 
\frac{1}{N|\zeta - P N^{-\alpha}(k^\alpha,0,0) |}.
\end{align*}
Then we obtain
\begin{align}
\Bigg| \sum_{n=0}^\infty \Big(\sum_{ \lambda\in\Lambda_{2n} } 
\frac{1}{N|\zeta - PN^{-\alpha}\lambda |} 
- \int_{ K_{2n}/N }^{ K_{2n+1}/N} \frac{dx}{|\zeta - P(x^\alpha,0,0)|}\Big)\Bigg|
\le \frac{2}{N D}.\label{ineq4}
\end{align}
The above inequality holds since 
the function $x\mapsto \frac{1}{|\zeta - P(x^\alpha,0,0)|}$ has at most one 
critical point and 
\begin{align*}
\sup_{x\in\R} \frac{1}{|\zeta - P(x^\alpha,0,0)|} 
- \inf_{x\in\R} \frac{1}{|\zeta - P(x^\alpha,0,0)|} \le \frac{1}{D}
\end{align*}
for all $\zeta\in K(R,D)$.
\end{proof}

Next we obtain the lower estimate of $\Phi_a$ as follows.
\begin{prop}
We have 
\begin{align}
\Phi_{S_n,P}^{T_n}(\zeta) 
&\ge\Big( \int_{S_n}^{T_n} \frac{dx}{1+ P x^\alpha}\Big) 
\min\Big\{ \frac{1}{|\zeta|},1\Big\},\label{est1}\\
\Phi_a(\zeta) 
&\ge \Big( \sum_{n=0}^\infty \int_{S_n}^{T_n} \frac{dx}{1+Px^\alpha} 
- 2 (aP^{-1})^\frac{1}{1+\alpha}\Big)\min\Big\{ \frac{1}{|\zeta|},1\Big\},
\label{est2}\\
\sum_{n=0}^\infty \Phi_{S_n,P}^{T_n}(\zeta) 
&\le P^{-\frac{1}{\alpha}} \frac{\alpha 2^\frac{1}{\alpha}}{\alpha - 1}
\frac{|\zeta|^\frac{1}{\alpha}}{|\zeta_\C|},\label{est3}\\ 
\sum_{n=n_0}^\infty \Phi_{S_n,P}^{T_n}(\zeta) 
&\le \frac{2 S_{n_0}^{-\alpha+1}}{P(\alpha - 1)}\quad 
\Big({\rm if}\ S_{n_0} \ge \Big(\frac{2|\zeta|}{P}\Big)^\frac{1}{\alpha}\Big),\label{est4}\\ 
\sum_{n=0}^{n_0} \Phi_{S_n,P}^{T_n}(\zeta) 
&\le \frac{T_{n_0}}{D}\quad ({\rm if}\ \zeta\in K(R,D)).\label{est5}
\end{align}
\label{lower1}
\end{prop}

\begin{proof}
First of all one can see 
\begin{align*}
\Phi_{S_n,P}^{T_n}(\zeta) 
\ge \int_{S_n}^{T_n} \frac{dx}{|\zeta|+Px^\alpha}
\ge \frac{1}{|\zeta|}\int_{S_n}^{T_n} \frac{dx}{1+Px^\alpha}
\end{align*}
if $|\zeta|\ge 1$, and  
\begin{align*}
\Phi_{S_n,P}^{T_n}(\zeta) 
\ge \int_{S_n}^{T_n} \frac{dx}{|\zeta|+Px^\alpha}
\ge \int_{S_n}^{T_n} \frac{dx}{1+Px^\alpha}
\end{align*}
if $|\zeta|\le 1$. 

Next we have
\begin{align*}
\Phi_a(\zeta) 
\ge \sum_{n=0}^\infty \sum_{k=K_{2n}}^{K_{2n+1}-1}
\frac{1}{N(|\zeta|+PN^{-\alpha}k^\alpha)}
\end{align*}
and the similar argument to the proof of Proposition \ref{conv1} gives 
\begin{align*}
\Bigg| \sum_{n=0}^\infty \Big(\sum_{k=K_{2n}}^{K_{2n+1}-1} 
\frac{1}{N(|\zeta|+PN^{-\alpha}k^\alpha)}
- \int_{S_n}^{T_n} \frac{dx}{|\zeta|+Px^\alpha} \Big) \Bigg| \le \frac{2}{N|\zeta|}.
\end{align*}
Combining these inequalities one can the second assertion if $|\zeta|\ge 1$. 
If $|\zeta|\le 1$, then we have 
\begin{align*}
\Phi_a(\zeta) 
\ge \sum_{n=0}^\infty \sum_{k=K_{2n}}^{K_{2n+1}-1} 
\frac{1}{N(1+PN^{-\alpha}k^\alpha)}
\end{align*}
and by the similar argument we obtain the assertion.

Next we consider \eqref{est3}.
If $t\ge (\frac{2|\zeta|}{P})^\frac{1}{\alpha}$, then 
\begin{align}
\int_t^\infty\frac{dx}{|\zeta-P(x^\alpha,0,0)|} 
\le \int_t^\infty\frac{2dx}{Px^\alpha} 
= \frac{2}{P(\alpha-1)}t^{-\alpha+1} \label{int_t}
\end{align}
holds. 
Hence one can see 
\begin{align*}
\sum_{n=0}^\infty \Phi_{S_n,P}^{T_n}(\zeta_\R,\zeta_\C) 
&\le \int_0^\infty \frac{dx}{|\zeta - P(x^\alpha,0,0)|} \\
&= \int_0^{(\frac{2|\zeta|}{P})^{\frac{1}{\alpha}}} \frac{dx}{|\zeta 
- P(x^\alpha,0,0)|} 
 + \int_{(\frac{2|\zeta|}{P})^{\frac{1}{\alpha}}}^\infty \frac{dx}{|\zeta 
- P(x^\alpha,0,0)|}\\
&\le \frac{(\frac{2|\zeta|}{P})^{\frac{1}{\alpha}}}{|\zeta_\C|} 
+ \frac{2}{P(\alpha -1)} 
\Big( \frac{2|\zeta|}{P} \Big)^{\frac{1}{\alpha}(-\alpha + 1)}\\
&= P^{-\frac{1}{\alpha}}\Big( \frac{(2|\zeta|)^{\frac{1}{\alpha}}}{|\zeta_\C|} 
+ \frac{(2|\zeta|)^{\frac{1}{\alpha}}}{\alpha -1} \frac{1}{|\zeta_\C|}
\Big).
\end{align*}
We have \eqref{est4} by \eqref{int_t}. 
\eqref{est5} is obvious. 
\end{proof}

Put 
\begin{align*}
A_{S,P}^T := \int_S^T \frac{dx}{1+Px^\alpha}.
\end{align*}

By Proposition \ref{lower1}, we have the following. 

\begin{prop}
Let $\Phi_a$ be as above. 
Then 
\begin{align*}
\sup_{|\zeta|\le R}\frac{1}{N\sqrt{\Phi_a(\zeta)}} 
\le \Big(\frac{a}{P}\Big)^\frac{1}{1+\alpha}
\Big( \sum_{n=0}^\infty A_{S_n,P}^{T_n} 
- 2 \Big(\frac{a}{P}\Big)^\frac{1}{1+\alpha}\Big)^{-\frac{1}{2}}\sqrt{R}
\end{align*}
holds for every $R\ge 1$.
\label{a2.1} 
\end{prop}

\section{Distance}\label{sec dis}
In the previous section we have estimated $|\Phi_a - \sum_{n}\Phi_{S_n,P}^{T_n}|$ 
from the above on $K(R,D)$. 

In this section we introduce more general positive functions 
$\Phi$ and $\Phi_\infty$, 
and induced metric $d,d_\infty$ on $\R^3$ respectively. 
What we hope to show in this section is that 
if we fix a very large $R\ge 1$ and assume that 
$\sup_{K(R,D)}|\Phi - \Phi_\infty|\le \frac{\varepsilon}{D}$ holds 
for a very small $\varepsilon$ and every $D\le 1$, then 
the identity map of $\R^3$ becomes the $(r,\delta)$-isometry 
from $(\R^3,d,0)$ to $(\R^3,d_\infty,0)$, for a large $r$ and a small $\delta$. 
Here, we explain the difficulty to show it.

We hope to show that $\sup_{K(R,D)}|d-d_\infty|$ 
is small for every $R\ge 1$ and $0<D\le 1$. 
By the estimate of $\sup_{K(R,D)}|\Phi - \Phi_\infty|$, 
it is easy to see that 
$\sup_{K(R,D)}|d_{R,D}-d_{\infty,R,D}|$ is small, where 
$d_{R,D}$ (resp. $d_{\infty,R,D}$) is the Riemannian distance of 
the Riemannian metric $\Phi h_0|_{K(R,D)}$ (resp. $\Phi_\infty h_0|_{K(R,D)}$). 
However, $d_{R,D}$ may not equal to $d$ in general 
since the geodesic of $\Phi h_0$ joining two points in $K(R,D)$ 
might leave from $K(R,D)$. 
To see that $\sup_{K(R,D)}|d_{R,D} - d|$ is sufficiently small, 
we have to observe that a path joining 
two points in $K(R,D)$ which leaves $K(R,D)$ can be 
replaced by a shorter path included in $K(R,D)$.

In this section we consider positive valued functions 
$\Phi,\Phi_\infty\in C^\infty(\R^3\backslash \mathbf{l})$ 
satisfying the following conditions for given constants 
$R\ge 1$, $m,\varepsilon,C_0,C_1>0$ and $\kappa \ge 0$.

\paragraph{(A3).}
\begin{align*}
|\Phi(\zeta)- \Phi_\infty(\zeta)| 
&\le \frac{\varepsilon}{D^{m}} \\
|\Phi(\zeta)- \Phi_\infty(\zeta)| 
&\le \frac{C_1}{D}
\end{align*}
holds for any $D\le 1$ and $\zeta\in K(R,D)$.

\paragraph{(A4).}
Along the decomposition $\R^3 = \R\oplus \C$, put 
$\zeta = (\zeta_\R,\zeta_\C) \in \R\oplus \C$. 
Then 
\begin{align*}
\Phi(\zeta_\R,e^{i\theta}\zeta_\C) &= \Phi(\zeta_\R,\zeta_\C),\quad 
\Phi(\zeta_\R,\zeta_\C) \le \Phi(\zeta_\R,\zeta_\C')\\
\Phi_\infty(\zeta_\R,e^{i\theta}\zeta_\C) &= \Phi_\infty(\zeta_\R,\zeta_\C),\quad 
\Phi_\infty(\zeta_\R,\zeta_\C) \le \Phi_\infty(\zeta_\R,\zeta_\C')
\end{align*}
holds for any $e^{i\theta}\in S^1$, if $|\zeta_\C| \ge |\zeta_\C'|$.

\paragraph{(A5).}
\[ \min\{ \Phi(\zeta), \Phi_\infty(\zeta) \} \ge
\left \{
\begin{array}{cl}
\frac{C_0}{|\zeta|} & ({\rm if}\ |\zeta|\ge 1),\\
C_0 & ({\rm if}\ |\zeta|\le 1).
\end{array}
\right.
\]

\paragraph{(A6).}
For any $u\ge 1$ and $\zeta\in\R^3\backslash \mathbf{l}$ 
with $|\zeta|\le u$, 
\begin{align*}
\Phi_\infty(\zeta) \le \frac{ C_1 u^{\kappa} }{|\zeta_\C|}
\end{align*}
holds.

\begin{rem}\normalfont
Let $\Phi=\Phi_a$ and $\Phi_\infty=\Phi_{S,P}^T$ 
be as in Section \ref{sec const}. 
Then they satisfy {\bf (A4)}, and also satisfy {\bf (A3)}{\bf (A5)}{\bf (A6)} 
for appropriate constants $\varepsilon,C_0,C_1$ given  
by Propositions \ref{conv1} and \ref{lower1}. 
\end{rem}

From now on, let $\Phi,\Phi_\infty$ satisfy {\bf (A3-6)} for 
constants $R,\varepsilon,C_0,C_1, \kappa$. 
Denote by $d,d_\infty$ the metric on $\R^3$ induced by 
$\Phi\cdot h, \Phi_\infty\cdot h$, and 
by $l,l_\infty$ the length of the path 
with respect to $d,d_\infty$, respectively.

\subsection{Estimates (1)}
Let $\mathbf{B}(u) := \{ \zeta\in\R^3;\ |\zeta| < u \}$ 
and ${\rm Path} (u,x,y)$ be the set of smooth paths in $\overline{\mathbf{B}(u)}$ 
joining 
$x,y\in \overline{\mathbf{B}(u)}$, 
then put 
\begin{align*}
d_u(x,y) &= \inf_{\gamma \in {\rm Path}(u,x,y)} l(\gamma),\\
d_{\infty,u}(x,y) &= \inf_{\gamma \in {\rm Path}(u,x,y)} l_\infty(\gamma),
\end{align*}
for $r$.
By the definition, $d(x,y) \le d_u(x,y)$ and $d_\infty(x,y) \le d_{\infty,u}(x,y)$ always hold. 
However, the opposite inequality may not hold, since 
the minimizing geodesic $\gamma$ joining $x,y\in \overline{\mathbf{B}(u)}$ 
may leave from $\overline{\mathbf{B}(u)}$. 
The goal of this subsection is to show 
$d_{\rho(u)}(x,y) \le d(x,y)$ and $d_{\infty,\rho(u)}(x,y) \le d_\infty(x,y)$ for 
a sufficiently large $\rho(u)$.

\begin{prop}
Suppose $\Phi,\Phi_\infty$ satisfy {\bf (A3-6)}. 
Let $D_{u}$ and $D_{u,u'}$ be the diameters of $\overline{\mathbf{B}(t)}$ 
with respect to $d$ and $d_{u'}$ respectively, where $0<u\le u'$. 
Define $D_{\infty,u}$ and $D_{\infty,u,u'}$ in the same way. 
Then the inequality 
\begin{align*}
2\sqrt{C_0} (\sqrt{|\zeta|}-1) \le \min \{ d(0,\zeta), d_\infty(0,\zeta)\}
\end{align*}
holds for all $\zeta\in\R^3$, 
and 
\begin{align*}
d(0,\zeta) &\le D_u \le D_{u,u} 
\le C_2 u^{\kappa'}\\
d_\infty(0,\zeta) &\le D_{\infty,u} \le D_{\infty,u,u} 
\le C_2 u^{\kappa'}
\end{align*}
hold for all $\zeta\in\R^3,u\ge 1$ with $|\zeta| \le u\le R$, where 
$C_2$ is the constant depending only on $C_1$ 
and $\kappa' = \frac{1+\kappa}{2}$. 
\label{diam}
\end{prop}
\begin{proof}
First of all we show the first inequality. 
Let $\gamma:[a,b]\to \R^3$ be a smooth path 
such that $\gamma(a) = 0$ and $\gamma(b) = \zeta$. 
We may suppose $|\zeta|\ge 1$, since it is obviously satisfied 
when $|\zeta|<1$.
Then there is $s\in [a,b]$ such that $|\gamma(s)| =1$ and $|\gamma(t)|\ge 1$ 
for any $t\in [s,b]$.
Then by the assumption {\bf (A5)}, one can see
\begin{align*}
l(\gamma) = \int_a^b\sqrt{\Phi(\gamma(t))} |\gamma'(t)| dt
\ge \int_s^b\sqrt{\Phi(\gamma)} |\gamma'| dt
\ge \int_s^b\sqrt{\frac{C_0}{|\gamma|}} |\gamma'| dt.
\end{align*}
Since we have $|\gamma'|\ge |\gamma|'$ holds, we obtain 
\begin{align*}
l(\gamma) \ge \int_s^b\sqrt{\frac{C_0}{|\gamma|}} |\gamma|' dt
\ge 2\sqrt{C_0}\int_s^b \frac{d}{dt}\sqrt{|\gamma|} dt
\ge 2\sqrt{C_0}(\sqrt{|\zeta|}-1) 
\end{align*}
for all $\zeta\in\R^3$ with $|\zeta|\ge 1$.

By the definition, $d(0,\zeta) \le D_u \le D_{u,R_1} \le D_{u,R_0}$ always hold 
for any $u\le R_0\le R_1$ and $\zeta\in\R^3$ with $|\zeta|\le u$. 
Next we estimate $D_{u,u}$ from the above. 
For every $\zeta$, we prepare the piecewise smooth paths 
$\gamma_{\zeta}$ in $\overline{\mathbf{B}(u)}$ joining $0$ and 
$\zeta$ as follows. 
Then we have an upper bound
\begin{align*}
D_{u,u} \le 2\sup_{\zeta\in \overline{\mathbf{B}(u)}} l(\gamma_{\zeta}).
\end{align*}
Here we define $\gamma_{\zeta}$ as follows. 
Now we have the isometric $S^1$-action on $\R^3$ 
with respect to $d,d_\infty$ by {\bf (A4)}. 
By supposing $\gamma_{e^{i\theta}\zeta} = e^{i\theta}\gamma_{\zeta}$, 
it suffices to consider $\gamma_{\zeta}$ in the case of 
$\zeta = r(\sin s,-\cos s,0)$, where $r>0$ and $-\pi< s\le \pi$. 
Let 
\begin{align*}
\gamma_{\zeta}|_{[0,1]}(t) &:= (0,-rt,0),\\
\gamma_{\zeta}|_{[1,2]}(t) &:= r(\sin\{ s(t-1)\}, -\cos\{ s(t-1)\},0).
\end{align*}
Since $\zeta \in K(R,|\zeta_\C|)$ holds, {\bf (A3)} gives 
$|\Phi(\zeta)- \Phi_\infty(\zeta)| 
\le \frac{C_1}{|\zeta_\C|}$, 
and {\bf (A6)} gives $\Phi_\infty(\zeta) \le C_1u^{\kappa}/|\zeta_\C|$.
Then we can see 
\begin{align*}
l(\gamma_{\zeta}|_{[0,1]}) &= \int_0^1\sqrt{\Phi(\gamma_{\zeta})}|\gamma_{\zeta}'|dt\\
&\le \int_0^1 r\sqrt{ |\Phi(\gamma_{\zeta})- \Phi_\infty(\gamma_{\zeta})| } dt
+ \int_0^1 r\sqrt{ |\Phi_\infty(\gamma_{\zeta})| } dt\\
&\le \int_0^1 r\sqrt{\frac{C_1}{ rt }} dt
+ \int_0^1 r\sqrt{ \frac{C_1u^{\kappa}}{ rt } } dt\\
& \le 2\sqrt{C_1 u} + 2\sqrt{C_1} u^{\frac{\kappa +1}{2} }.
\end{align*}
Simultaneously, we also have 
\begin{align*}
l(\gamma_{\zeta}|_{[1,2]}) 
&\le \int_1^2 |\gamma_{\zeta}'| \sqrt{ |\Phi(\gamma_{\zeta})- \Phi_\infty(\gamma_{\zeta})| } dt
+ \int_1^2 |\gamma_{\zeta}'| \sqrt{ |\Phi_\infty(\gamma_{\zeta})| } dt\\
&\le \int_0^1 r|s|\sqrt{\frac{C_1}{ r|\cos st| } } dt
+ \int_0^1 r|s|\sqrt{ \frac{C_1u^{\kappa}}{ r|\cos st| } } dt\\
&\le \sqrt{C_1 u} +\sqrt{C_1 u^{1+\kappa}}
\int_0^{|s|} \sqrt{\frac{1}{ \cos t } } dt.
\end{align*}
Here, $\int_0^{\pi} \sqrt{\frac{1}{ \cos t }} dt$ 
is finite.  
Since $u\ge 1$ and $\kappa\ge 0$, 
we may suppose $\max\{ \sqrt{u},\sqrt{u^{1+\kappa}}\} = u^{1+\kappa}$. 
By combining these estimates and putting 
$C_2 = (2+ \int_0^{\pi} \sqrt{\frac{1}{ \cos t }} dt)\sqrt{C_1}$, 
we have the assertion. 
The estimate of $D_{\infty,u,u}$ also obtained by the above argument. 
\end{proof}

\begin{prop}
Suppose $\Phi,\Phi_\infty$ satisfy {\bf (A3-6)}, and let 
\begin{align*}
\rho(t):= \max\{ t-1, ( 1+C_3 t^{\kappa'} )^2\}, 
\end{align*}
for $t>0$, 
where $C_3 = \frac{3C_2}{2\sqrt{C_0}}$ and 
$C_2$ is the constant in Proposition \ref{diam}. 
Then $d_{\rho(u)}(x,y) = d(x,y)$ and 
$d_{\infty,\rho(u)}(x,y) = d_\infty(x,y)$ holds 
for any $x,y\in \mathbf{B}(u)$ and $1\le u \le R$. 
\label{ball}
\end{prop}
\begin{proof}
By the definition, $d(x,y) \le d_{\rho(u)}(x,y)$ always holds. 
We assume $d(x,y) < d_{\rho(u)}(x,y)$ for some 
$x,y\in \mathbf{B}(u)$.
Then there is a smooth $\gamma:[a,b]\to\R^3$ joining $x$ and $y$ such that 
$d(x,y)\le l(\gamma)<d_{\rho(u)}(x,y)$, 
which implies the existence of $c\in [a,b]$ satisfying $|\gamma(c)| = \rho(u)$. 
Then one can see 
\begin{align*}
l(\gamma) \ge l(\gamma|_{[a,c]})
&\ge d(0,\gamma(c)) - d(0,\gamma(a))\\
&\ge 2\sqrt{C_0} (\sqrt{\rho (u)} -1) - D_{u,u}\\
&\ge 2\sqrt{C_0} (\sqrt{( 1+C_3 u^{\kappa'} )^2} -1) 
- C_2 u^{\kappa'}\\
&\ge 2 C_2 u^{\kappa'}
\end{align*}
by Proposition \ref{diam}.
On the other hand, we have 
\begin{align*}
d_{\rho(u)}(x,y) \le D_{u,\rho(u)} \le D_{u,u} 
\le C_2 u^{\kappa'}
\end{align*}
by Proposition \ref{diam}. 
Therefore we obtain 
\begin{align*}
2 C_2 u^{\kappa'} 
\le l(\gamma)<d_{\rho(u)}(x,y) \le C_2 u^{\kappa'}, 
\end{align*}
we have a contradiction. 
$d_\infty(x,y) = d_{\infty,\rho(u)}(x,y)$ is also shown in the same way. 
\end{proof}

\subsection{Estimates (2)}\label{estimate2}
In this subsection, let $\gamma:[a,b]\to \mathbf{B}(u)$ be a smooth path 
joining $x,y\in \R^3\backslash L(D)$ where 
\begin{align*}
L(D):=\{ \zeta\in\R^3;\ |\zeta_\C| < D\}. 
\end{align*}
Now, we are going to show that 
if $\gamma$ is a minimizing geodesic joining $x$ and $y$, 
then it never approaches to the axis $\{ (t,0,0)\in\R^3;t\in\R\}$. 
To show it, if the given $\gamma$ invade $L(D)$, 
then we modify $\gamma$ and construct the new path $c_\gamma$ 
not to invade $L(D)$.

\begin{lem}
Suppose $\Phi,\Phi_\infty$ satisfy {\bf (A4)}.
Let $\gamma=(\gamma_\R,\gamma_\C):[a,b]\to \R^3=\R\oplus\C$ be a smooth path satisfying that 
$|\gamma_\C(a)| = |\gamma_\C(b)| = D$ and $|\gamma_\C(t)| \le D$ for 
any $t\in [a,b]$.
Define $P_\gamma:[a,b]\to\R^3$ by 
\begin{align*}
P_\gamma(t):=(\gamma_\R(t), \gamma_\C(a)).
\end{align*}
Then $l(P_\gamma) \le l(\gamma)$ and 
$l_\infty(P_\gamma) \le l_\infty(\gamma)$ hold. 
\label{lowreplace}
\end{lem}
\begin{proof}
Since $\Phi(\gamma(t)) \ge \Phi(P_\gamma(t))$ holds 
by the second inequality of {\bf (A4)}, 
and 
\begin{align*}
|\gamma'|^2=|\gamma_\R'|^2 + |\gamma_\C'|^2 \ge |\gamma_\R'|^2 
= |P_{\gamma}'|^2
\end{align*}
holds, we can deduce 
\begin{align*}
l(\gamma) = \int_a^b\sqrt{\Phi(\gamma(t))}|\gamma'(t)|dt 
\ge \int_a^b\sqrt{\Phi(P_\gamma(t))}|P_{\gamma}'(t)|dt. 
\ge l(P_{\gamma}).
\end{align*}
\end{proof}

Let $\gamma:[a,b] \to \R^3$ be a smooth path 
joining $x,y\in \R^3\backslash L(D)$, 
and assume that $|\gamma_\C(a')|=|\gamma_\C(b')| = D$ 
and $\gamma((a',b'))$ is contained in $\overline{L(D)}$ 
for some $a\le a'<b'\le b$. 
Then define a new path $\Gamma(\gamma,[a',b']):[a,b] \to \R^3$ 
by connecting 
\begin{align*}
\gamma|_{[a,a']},\quad P_{\gamma|_{[a',b']}},\quad e^{i\theta}\gamma|_{[b',b]}.
\end{align*}
Here, by choosing $e^{i\theta}$ appropriately, $\Gamma(\gamma,[a',b'])$ 
is the continuous and piecewise smooth. 
By Lemma \ref{lowreplace}, the length of 
$\Gamma(\gamma,[a',b'])$ is not longer than 
that of $\gamma$ 
since $S^1$ rotation preserves $d$ and $d_\infty$.

Put $J:= \gamma^{-1}(L(D))\cap (a,b)$. 
Since $J$ is open in $(a,b)$, 
it is decomposed into the disjoint open intervals 
such as 
\begin{align*}
J = \bigsqcup_{q\in \mathcal{Q}} (a_q,b_q)
\end{align*}
respectively 
for some $a_q,b_q\in [a,b]$ and countable sets $\mathcal{Q}$. 
If $q\in\mathcal{Q}$, then $|\gamma_\C(a_q)|=|\gamma_\C(b_q)|=D$ holds. 
Then we have $\gamma_1:= \Gamma(\gamma,[a_q,b_q])$ for a fixed $q\in\mathcal{Q}$, 
moreover we obtain $\gamma_2:=\Gamma(\gamma_1,[a_{q'},b_{q}])$ for 
another $q'\in\mathcal{Q}$, and repeating this process 
for all $q\in \mathcal{Q}$ 
we finally obtain the piecewise smooth path $c:[a,b]\to \R^3$ 
such that $c(a) = \gamma(a)$, 
$c(b) = e^{i\theta}\gamma(b)$ for some $e^{i\theta_0}$ and 
\begin{align*}
l(c)\le l(\gamma),\quad 
l_\infty(c)\le l_\infty(\gamma).
\end{align*}
Here, we have to modify $c$ so that 
the terminal points of both paths coincides. 
Put $\overline{b}:=\sup\{ t\in [a,b];\ |\gamma_\C(t)| =D \}$. 
Then define a path $\hat{\gamma}$ by 
connecting $c|_{[a,\overline{b}]}$ and 
$\gamma|_{[\overline{b},b]}$. 
Here, to connect $c(\overline{b})$ and $\gamma(\overline{b})$, 
we add the path $c_{\theta_0}:[0,\theta_0]\to \partial L(D)$ 
defined by $c_{\theta_0}(t)=e^{it}\gamma(\overline{b})$. 
Then by {\bf (A6)}, we obtain 
$l(c_{\theta_0})\le \sqrt{C_1(1 +  u^{\kappa})}\sqrt{D}$ 
and $l_\infty(c_{\theta_0})\le \sqrt{C_1 u^{\kappa}}\sqrt{D}$
if $|\gamma(\overline{b})|\le u \le R$. 
Hence we have the following proposition.

\begin{prop}\label{prop appro}
Let $D\le 1$ and $1\le u\le R$, and 
$x,y,\gamma, \hat{\gamma}$ be as above. 
If the image of $\gamma$ is contained in $\mathbf{B}(u)$, 
then we have
\begin{align*}
l(\hat{\gamma}) - l(\gamma) 
&\le \sqrt{C_1(1 +  u^{\kappa})}\sqrt{D},\\
l_\infty(\hat{\gamma}) - l_\infty(\gamma) 
&\le \sqrt{ C_1 u^{\kappa}}\sqrt{D}.
\end{align*}
\end{prop}

\begin{prop}
Let $x,y,\gamma, \hat{\gamma}$ be as above. 
If the image of $\gamma$ is contained in $\mathbf{B}(u)$, 
then the image of $\hat{\gamma}$ is contained in 
$\mathbf{B}(u + D)\backslash L(D)$.
\label{image of path}
\end{prop}
\begin{proof}
It is obvious by the construction that 
the image of $\hat{\gamma}$ is contained in $\R^3\backslash L(D)$.

Since $S^1$-action preserves $\mathbf{B}(u)$ and  
\begin{align*}
|P_{\gamma}|^2 \le |\gamma|^2 + D^2,\quad 
\Big|\Big(\gamma_\R(t), \frac{ D \gamma_\C(t)}{|\gamma_\C(t)|}\Big)\Big|^2
\le |\gamma|^2 + D^2
\end{align*}
holds, we have th assertion.
\end{proof}

\subsection{Estimates (3)}
Let 
\begin{align*}
{\rm Path}(u,D,x,y) &:= \{ \gamma\in{\rm Path}(x,y);\ {\rm Im}(\gamma) \subset K(u,D)\},\\
d_{u,D}(x,y) &:= \inf_{\gamma\in {\rm Path}(u,D,x,y)} l(\gamma)\\
d_{\infty,u,D}(x,y) &:= \inf_{\gamma\in {\rm Path}(u,D,x,y)} l_\infty(\gamma).
\end{align*}
for $x,y\in K(u,D)$.
By the definition, $d(x,y)\le d_{u,D}(x,y)$ always holds. 
In this subsection we consider the opposite estimate.

\begin{lem}
Let $\hat{\zeta}:=(\zeta_\R,\frac{D\zeta_\C}{|\zeta_\C|})$ if 
$\zeta_\C\neq 0$, 
and $\hat{\zeta}:=(\zeta_\R,D)$ if 
$\zeta_\C= 0$. 
Suppose $\Phi,\Phi_\infty$ satisfy {\bf (A3-6)}, and $1\le u\le R$. 
\begin{itemize}
\setlength{\parskip}{0cm}
\setlength{\itemsep}{0cm}
 \item[(1)] If $\zeta\in L(D)\cap \mathbf{B}(u-1)$ and $0< D\le 1$,
then 
\begin{align*}
d_{u}(\zeta, \hat{\zeta}) \le 2\sqrt{C_1(1 + u^{\kappa}) D},
\quad 
d_{\infty,u}(\zeta, \hat{\zeta}) \le 2\sqrt{ C_1 u^\kappa D}
\end{align*}
hold.
 \item[(2)] If $\zeta\in L(D)\cap K(u-1,D)$ and $0< D\le 1$, 
then 
\begin{align*}
d_{u,D}(\zeta, \hat{\zeta}) \le 2\sqrt{C_1(1 + u^{\kappa}) D},
\quad 
d_{\infty,u,D}(\zeta, \hat{\zeta}) \le 2\sqrt{ C_1 u^\kappa D}
\end{align*}
hold. 
\end{itemize}
\label{7.11}
\end{lem}


\begin{proof}
Let $\gamma(t)=(\zeta_\R,t\hat{\zeta}_\C)$ for $t\in [|\zeta_\C|/D,1]$. 
Then $\gamma$ is joining $\zeta$ and $\hat{\zeta}$, and the 
image of $\gamma$ is contained in $\mathbf{B}(u-1+D)
\subset \mathbf{B}(u)$. 
Then by {\bf (A3,6)} we have 
$\Phi(\gamma(t)) \le C_1(1 + u^{\kappa})/(tD)$. 
Then we have 
\begin{align*}
d_{u}(\zeta,\hat{\zeta}) \le l(\gamma)
\le 2\sqrt{C_1(1 + u^{\kappa}) D}.
\end{align*}
Moreover, if $\zeta\in K(u-1,D)$, then the image of $\gamma$ 
is contained in $K(u,D)$, therefore we have 
\begin{align*}
d_{u,D}(\zeta,\hat{\zeta}) \le l(\gamma)
\le 2\sqrt{C_1(1 + u^{\kappa})D}.
\end{align*}
The estimates for $d_{\infty,u}(\zeta,\hat{\zeta})$ 
and $d_{\infty,u,D}(\zeta,\hat{\zeta})$ follows in the same way.
\end{proof}

\begin{prop}
Suppose $\Phi,\Phi_\infty$ satisfy {\bf (A3-6)} and  
let $\rho$ be as in Proposition \ref{ball}. 
If $\rho(u+1)+1\le R$, 
then
\begin{align*}
|d_{\rho(u+1)+1,D}(x,y) - d(x,y)| &\le \xi (u) \sqrt{D}\\
|d_{\infty,\rho(u+1)+1,D}(x,y) - d_\infty(x,y)| &\le \xi_\infty (u) \sqrt{D}
\end{align*}
holds for any $x,y\in K(u,D)$ and $0<D\le 1$, 
where 
\begin{align*}
\xi (u) &:= \sqrt{C_1(1 + (\rho(u+1)+1)^\kappa)} 
+ 8\sqrt{C_1(1 + (u+1)^\kappa)}+2,\\
\xi_\infty (u) &:= \sqrt{C_1(\rho(u+1)+1)^\kappa} 
+ 8\sqrt{C_1(u+1)^\kappa}+2.
\end{align*}
\label{diff of distance}
\end{prop}
\begin{proof}
Since $d(x,y) \le d_{\rho(u+1)+1,D}(x,y)$ always holds, 
it suffices to show that 
$d_{\rho(u+1)+1,D}(x,y) - d(x,y) \le \xi(u)\sqrt{D}$. 
Let $x,y\in K(u,D)$ and $0< D\le 1$. 
By the assumption $\rho(u+1)+1\le R$ and the definition of 
$\rho$, $u+1\le R$ holds. 
Define $\hat{x}\in\R^3$ as in Lemma \ref{7.11} if $x\in L(D)$, 
and $\hat{x}:= x$ if $x\notin L(D)$. 
Define $\hat{y}$ in the same way.
Then we can see $\hat{x},\hat{y}\in \mathbf{B}(u + 1)\backslash L(D)$ 
and $d_{u+1,D}(x,\hat{x}) \le 2\sqrt{C_1(1 +  (u+1)^\kappa) D}$ 
by Lemma \ref{7.11}, 
consequently we obtain 
\begin{align}
d_{u+1,D}(x,\hat{x}) + d_{u+1,D}(y,\hat{y}) 
\le 4\sqrt{C_1(1 +  (u+1)^\kappa) D}.\label{ineq small}
\end{align}

For any $\gamma\in {\rm Path}(\hat{x},\hat{y})$, we construct 
$F(\gamma)\in {\rm Path}(\rho(u+1)+1,D,\hat{x},\hat{y})$ as follows. 
By the Proposition \ref{ball}, we can see
\begin{align*}
l(\gamma)\ge d(\hat{x},\hat{y}) = d_{\rho(u+1)}(\hat{x},\hat{y}) 
= \inf_{c\in {\rm Path}(\rho(u+1),x,y)}l(c),
\end{align*}
accordingly we can take 
$c\in {\rm Path}(\rho(u+1),x,y)$ such that 
$l(c)\le l(\gamma) +\sqrt{D}$. 
Then we can apply the argument in Section \ref{estimate2} to 
$\hat{x},\hat{y}$ and $c$ so that 
we obtain piecewise smooth paths $\hat{c}$ whose image is contained in 
$\mathbf{B}(\rho(u+1) + 1)\backslash L(D)$, 
hence in $K(\rho(u+1) + 1,D)$.
Then we have 
\begin{align*}
\liminf_{k\to\infty}l(\hat{c}) - l(c) \le \sqrt{C_1(1 + (\rho(u+1)+1)^\kappa)D},
\end{align*}
by Proposition \ref{prop appro}. 
Therefore, there is a sufficiently large $k$, which may depend on $n$ and $D$, 
such that $l(\hat{c}) - l(c) 
\le \sqrt{C_1(1 + (\rho(u+1)+1)^\kappa)D} + \sqrt{D}$. 
Put $F(\gamma) = \hat{c}$. 
Then we can see 
\begin{align*}
l(F(\gamma)) - l(\gamma) 
&\le l(F(\gamma)) - l(c) + l(c) - l(\gamma) \\
&\le \sqrt{C_1(1 + (\rho(u+1)+1)^\kappa)D} + \sqrt{D} + \sqrt{D}\\
&= \{ \sqrt{C_1(1 + (\rho(u+1)+1)^\kappa)}+2\}\sqrt{D}.
\end{align*}
Thus we obtain $F(\gamma)\in {\rm Path}(\rho(u+1)+1,D,\hat{x},\hat{y})$ for 
every $\gamma\in {\rm Path}(\hat{x},\hat{y})$, such that 
\begin{align}
l( F(\gamma) ) - l(\gamma) 
\le \{ \sqrt{C_1(1 + (\rho(u+1)+1)^\kappa)}+2\}\sqrt{D} \label{ineq 18}.
\end{align}
By taking the infimum of \eqref{ineq 18} for all 
$\gamma\in {\rm Path}(\hat{x},\hat{y})$, 
we obtain 
\begin{align}
d_{\rho(u+1)+1,D}(\hat{x},\hat{y}) \le d(\hat{x},\hat{y}) 
+ \{ \sqrt{C_1(1 + (\rho(u+1)+1)^\kappa)}+2\}\sqrt{D}.\label{ineq 19}
\end{align}
Since $\rho(u+1)\ge u+1$, we have
\begin{align*}
|d_{\rho(u+1)+1,D}(\hat{x},\hat{y}) - d_{\rho(u+1)+1,D}(x,y)| 
&\le d_{\rho(u+1)+1,D}(\hat{x},x) + d_{\rho(u+1)+1,D}(\hat{y},y)\\
&\le d_{u+1,D}(\hat{x},x) + d_{u+1,D}(\hat{y},y)\\
&\le 4\sqrt{C_1(1 +  (u+1)^\kappa) D}\\
|d(\hat{x},\hat{y}) - d(x,y)| 
&\le d(\hat{x},x) + d(\hat{y},y)\\
&\le d_{u+1,D}(\hat{x},x) + d_{u+1,D}(\hat{y},y)\\
&\le 4\sqrt{C_1(1 +  (u+1)^\kappa) D}
\end{align*}
by \eqref{ineq small}, hence 
\begin{align*}
d_{\rho(u+1)+1,D}(x,y) 
&\le d_{\rho(u+1)+1,D}(\hat{x},\hat{y})
+ 4\sqrt{C_1(1 +  (u+1)^\kappa) D}\\
d(\hat{x},\hat{y}) &\le d(x,y) + 4\sqrt{C_1(1 +  (u+1)^\kappa) D}
\end{align*}
hold. 
By combining these inequalities with \eqref{ineq 19}, we obtain 
\begin{align*}
d_{\rho(u+1)+1,D}(x,y) \le d(x,y) + \xi (u)\sqrt{D}.
\end{align*}
The second inequality can be shown in the same way. 
\end{proof}

\subsection{From {\bf (A3-6)} to {\bf (A1-2)}}\label{convergence}

\begin{prop}
Suppose that $\Phi,\Phi_\infty$ satisfy {\bf (A3-6)}, 
and let $\gamma:[a,b]\to K(u,D)$ and $1\le u\le R$.
Then 
\begin{align*}
|l(\gamma) - l_\infty (\gamma)| \le \sqrt{\frac{\varepsilon u}{C_0 D^{m}}} 
l_\infty (\gamma).
\end{align*}
holds.
\label{prop length}
\end{prop}
\begin{proof}
Since $l(\gamma) = \int_a^b\sqrt{\Phi(\gamma(t))} |\gamma'(t)|dt$, 
one can see
\begin{align*}
|l(\gamma) - l_\infty (\gamma)| &\le \int_a^b\sqrt{|\Phi(\gamma) - \Phi_\infty(\gamma) |} |\gamma'|dt\\
&\le \int_a^b\sqrt{\frac{|\Phi(\gamma) - \Phi_\infty(\gamma) |}{\Phi_\infty(\gamma)}} \sqrt{\Phi_\infty(\gamma)}|\gamma'|dt\\
&\le \int_a^b\sqrt{\frac{\varepsilon \max\{ |\gamma|,1\}}{C_0 D^{m}}} \sqrt{\Phi_\infty(\gamma(t))}|\gamma'(t)|dt
\end{align*}
by {\bf (A3,5)}.
Since we have assumed $|\gamma|\le u$ and $u\ge 1$, we have 
\begin{align*}
|l(\gamma) - l_\infty (\gamma)| \le \sqrt{\frac{\varepsilon u}{C_0 D^{m}}} 
l_\infty (\gamma).
\end{align*}
\end{proof}

\begin{prop}
Suppose that $\Phi,\Phi_\infty$ satisfy {\bf (A3,5,6)}, 
then,
\begin{align*}
|d_{u,D}(x,y) - d_{\infty,u,D}(x,y)|
\le \sqrt{\frac{\varepsilon u}{C_0 D^{m}}} d_{\infty,u,D}(x,y)
\end{align*}
holds for all $1\le u\le R$.
\label{distance on compact set}
\end{prop}
\begin{proof}
Put $\delta = \sqrt{\frac{\varepsilon u}{C_0 D^{m}}}$. 
Then Proposition \ref{prop length} gives 
\begin{align}
(1- \delta) l_\infty(\gamma)
\le l(\gamma)
\le (1+ \delta ) l_\infty(\gamma).\label{ineq length}
\end{align}
Then by taking the infimum of \eqref{ineq length} for all 
$\gamma\in {\rm Path}(u,D,x,y)$, we can see 
\begin{align*}
(1- \delta)d_{\infty,u,D}(x,y)
\le d_{n,u,D}(x,y)
\le (1+ \delta)d_{\infty,u,D}(x,y)
\end{align*}
for all $u\ge 0$.
\end{proof}

\begin{prop}
Suppose that $\Phi,\Phi_\infty$ satisfy {\bf (A3-6)} $u\le 1$, 
and let $u^{(2)}:=\rho(u+2) + 1 \le R$. 
Then we have 
\begin{align*}
|d(x,y) - d_\infty (x,y)| &\le 
26\sqrt{C_1(1 + R^\kappa) D} + 4 \sqrt{D} \\
&\quad\quad + \sqrt{\frac{\varepsilon R}{C_0 D^{m}}}
\{ C_2 R^{\kappa'} + (9\sqrt{C_1 R^\kappa} + 2) \sqrt{D} \}.
\end{align*}
for all $x,y\in \mathbf{B}(u)$. 
\label{key prop}
\end{prop}
\begin{proof}
Put $u^{(1)}=\rho(u+1)+1$ and let $x,y\in K(u,D)$. 
Then $u^{(1)}\le R$. 
By combining Propositions \ref{diff of distance} 
and \ref{distance on compact set}, 
we have 
\begin{align*}
|d(x,y) - d_\infty(x,y)| 
&\le |d(x,y) - d_{u^{(1)},D}(x,y)| + |d_{u^{(1)},D}(x,y) - d_{\infty,u^{(1)},D}(x,y)|\\
&\quad \quad + |d_\infty(x,y) - d_{\infty,u^{(1)},D}(x,y)|\\
&\le \xi(u) \sqrt{D} + \xi_\infty(u) \sqrt{D} 
+ \sqrt{\frac{\varepsilon u^{(1)}}{C_0 D^{m}}}d_{\infty,u^{(1)},D}(x,y)\\
&\le 2\xi(u) \sqrt{D} + \sqrt{\frac{\varepsilon u^{(1)}}{C_0 D^{m}}}
(d_{\infty}(x,y) + \xi_\infty(u) \sqrt{D}).
\end{align*}
By the Proposition \ref{diam}, 
$D_{\infty,u} < C_2 u^{\kappa'}$ holds if $u\ge 1$, 
consequently $d_{\infty}(x,y)$ is not more than $C_2 u^{\kappa'}$. 
Therefore we obtain 
\begin{align*}
|d(x,y) - d_\infty(x,y)| 
\le 2\xi(u) \sqrt{D} + \sqrt{\frac{\varepsilon u^{(1)}}{C_0 D^{m}}}
(C_2 u^{\kappa'} + \xi_\infty(u) \sqrt{D})
\end{align*}
for all $x,y\in K(u,D)$. 

Next we consider the case of $x\in \mathbf{B}(u)$ but not contained in $K(u,D)$. 
In this case $x\in \mathbf{B}(u)\cap L(D)$ holds, hence we can apply 
Lemma \ref{7.11}. 
Let $\hat{x}$ be as in Lemma \ref{7.11}. 
Then we can see that 
\begin{align*}
d(x,\hat{x}) \le 2\sqrt{C_1(1 +  (u+1)^\kappa) D}
\end{align*}
and 
$\hat{x}$ is contained in $K(u+1,D)$. 
Here we suppose that $y$ is also contained 
in $\mathbf{B}(u)\cap L(D)$, and 
follow the same procedure. 
If $y$ is in $K(u,D)$, then suppose $y=\hat{y}$ in the following discussion.
Now we have 
\begin{align*}
|d(x,y) - d(\hat{x},\hat{y})| \le d(x,\hat{x}) + d(y,\hat{y}) 
\le 4\sqrt{C_1(1 +  (u+1)^\kappa) D},
\end{align*}
hence we can see 
\begin{align*}
|d(x,y) - d_\infty(x,y)| &\le 8\sqrt{C_1(1 +  (u+1)^\kappa) D} 
+ |d(\hat{x},\hat{y}) - d_\infty(\hat{x},\hat{y})|\\
&\le 8\sqrt{C_1(1 +  (u+1)^\kappa) D} + 2\xi(u+1) \sqrt{D} \\
&\quad\quad + \sqrt{\frac{\varepsilon u^{(2)}}{C_0 D^{m}}}
\{ C_2 (u+1)^{\kappa'} + \xi_\infty(u+1) \sqrt{D} \}. 
\end{align*}
Since $\xi(u)$ is monotonically increasing and 
$u+2 \le u^{(2)} \le R$ holds, we have 
\begin{align*}
\xi(u+1) \le 9\sqrt{C_1(1 + R^\kappa)} + 2,\quad 
\xi_\infty(u+1) \le 9\sqrt{C_1 R^\kappa} + 2
\end{align*}
\end{proof}

\begin{cor}
Suppose that $\Phi,\Phi_\infty$ satisfy {\bf (A3-6)} and 
$\varepsilon\le 1$, 
and let $u^{(2)}:=\rho(u+2) + 1 \le R$. 
Then there exists a constant $C$ independent of any other constants 
such that 
\begin{align*}
|d(x,y) - d_\infty (x,y)| &< 
C(1+\sqrt{C_1})( 1+C_0^{-\frac{1}{2}} )R^{1+\frac{\kappa}{2}}\varepsilon^\frac{1}{2(1+m)}.
\end{align*}
for all $x,y\in \mathbf{B}(u)$. 
\label{key cor}
\end{cor}
\begin{proof}
In Proposition \ref{key prop}, 
let $D=\varepsilon^\frac{1}{1+m} \le 1$. 
As described in the proof of Proposition \ref{diam}, 
$C_2$ is linearly depending on $\sqrt{C_1}$. 
Then assertion follows by using $R\ge 1$, 
$\varepsilon\le 1$ and unifying constants. 
\end{proof}

\begin{prop}
Suppose that $\Phi(\zeta)\ge \frac{A}{|\zeta|}$ holds 
for some $A>0$ and all $\zeta$ with $|\zeta|\le 1$, and 
let $u(r) := ( 1 + \frac{A^{-\frac{1}{2}}r}{2} )^2$. 
Then $B(0,r) \subset \mathbf{B}(u(r))$
holds for all $r>0$, 
where $B(0,r)$ are the metric ball 
with respect to $d$. 
\label{euclidean ball}
\end{prop}
\begin{proof}
Let $\zeta\in B(0,r)$. 
Then by the same argument in the proof of 
the first inequality of Proposition \ref{diam} we have
\begin{align*}
2\sqrt{A}(\sqrt{|\zeta|} -1)\le d(0,\zeta) < r,
\end{align*}
which gives $|\zeta| < (1 + \frac{A^{-\frac{1}{2}}r}{2})^2 = u(r)$. 
\end{proof}

\begin{prop}\label{prop isom}
Suppose that $\Phi,\Phi_\infty$ satisfy {\bf (A3-6)} 
and suppose $\varepsilon \le 1$. 
Then the identity map of $\R^3$ is $(r,\delta)$-isometry 
from $(\R^3,d,0)$ to $(\R^3,d_\infty,0)$, where $r,\delta>0$ are defined by 
$$
\rho(u(r)+2) +1=R,\quad 
\delta = C (1+\sqrt{C_1})( 1+C_0^{-\frac{1}{2}} )R^{1+\frac{\kappa}{2}}\varepsilon^\frac{1}{2(1+m)}.
$$
\end{prop}
\begin{proof}
Let $x,y\in B(0,r)$. 
Then $x,y\in \mathbf{B}(u(r))$, hence 
\begin{align}
|d(x,y)-d_\infty(x,y)| 
< C (1+\sqrt{C_1})( 1+C_0^{-\frac{1}{2}} )R^{1+\frac{\kappa}{2}}\varepsilon^\frac{1}{2(1+m)} \label{e-isom}
\end{align} 
holds.  
Next we show $B_\infty(0,r-\delta) \subset B(B(0,r),\delta)$. 
If $x\in B_\infty(0,r-\delta)$, then $x\in\mathbf{B}(u(r))$ holds, 
therefore \eqref{e-isom} gives 
\begin{align*}
d(0,x)
&< d_\infty(0,x) + C (1+\sqrt{C_1})( 1+C_0^{-\frac{1}{2}} )
R^{1+\frac{\kappa}{2}}\varepsilon^\frac{1}{2(1+m)} \\
&< r-\delta + C (1+\sqrt{C_1})( 1+C_0^{-\frac{1}{2}} )
R^{1+\frac{\kappa}{2}}\varepsilon^\frac{1}{2(1+m)} =r,
\end{align*} 
which implies $B_\infty(0,r-\delta) \subset B(0,r)$.
\end{proof}

By Propositions \ref{euclidean ball}, 
and Proposition \ref{a2.1}, the following estimate is obtained.

\begin{prop}\label{fiberdiam}
Let $\Phi_a$ be as in Section \ref{sec const} and assume 
$\sum_{n=0}^\infty A_{S_n,P}^{T_n} - 2a^\frac{1}{1+\alpha}>0$. 
Then $\sup_{\zeta\in B(0,r)}\frac{1}{N\sqrt{\Phi_a(\zeta)}}$ 
is not more than 
\begin{align*}
\frac{ (\frac{a}{P} )^\frac{1}{1+\alpha}}
{\sqrt{ \sum_{n=0}^\infty A_{S_n,P}^{T_n} 
- 2 (\frac{a}{P} )^\frac{1}{1+\alpha} } }
\Big( 1+ \frac{r}{2\sqrt{\sum_{n=0}^\infty A_{S_n,P}^{T_n}  - 2 (\frac{a}{P} )^\frac{1}{1+\alpha}}}\Big).
\end{align*}
\end{prop}

Combining Propositions 
\ref{prop isom} and \ref{fiberdiam}, 
we obtain the following theorem.

\begin{thm}\label{key}
Let $a_i,P_i,n_i>0$, $\lim_{i\to \infty}a_i = 0$ and 
$\lim_{i\to \infty}n_i\to\infty$. 
Put $S_{i,n_i}, T_{i,n_i}$ as in Section \ref{sec const}. 
Suppose that there are constants 
$\varepsilon=\varepsilon_i(R)$, $C_0$, 
$C_1$, $\kappa$, $m$ 
for all $R\ge 1$ such that 
$\Phi=\Phi_{a_i}$ and $\Phi_\infty$ satisfy {\bf (A3-6)}. 
If 
\begin{align*}
\lim_{i\to\infty}\varepsilon_i (R) = \lim_{i\to\infty}\frac{a_i}{P_i} = 0,\quad 
\liminf_{i\to\infty}\sum_{l=0}^\infty A_{S_{i,n_i},P_i}^{T_{i,n_i}} >0
\end{align*}
and $C_0,C_1,\kappa,m$ are independent of 
$i,R$, 
then 
\begin{align*}
\{ (X,a_ig_\Lambda,p)\}_i \xrightarrow[i\to\infty]{GH} (\R^3,d_\infty,0).
\end{align*}
\end{thm}
\begin{proof}
Fix $r>0$ and $\delta>0$ arbitrarily. 
Put 
$R(r)=\rho(u(r)+2) +1$, and let $C>0$ be the constant 
in Corollary \ref{key cor}. 
By the assumption, there exists $i(r,\delta)>0$ such that 
$C(1+\sqrt{C_1})(1+C_0^{-\frac{1}{2}})
R(r)^{1+\frac{\kappa}{2}}\varepsilon_i(R(r))^\frac{1}{2(1+m)} < \frac{\delta}{2}$ 
holds for all $i\ge i(r,\delta)$. 
Then by Proposition \ref{prop isom}, ${\rm id}_{\R^3}$ is the 
$(r,\frac{\delta}{2})$-isometry from $(\R^3,d_{a_i},0)$ to 
$(\R^3,d_{a_\infty},0)$. 
By Proposition \ref{fiberdiam}, we can take $i'(r,\delta)\ge i(r,\delta)$ 
such that 
$\sup_{\zeta\in B(0,r)}\frac{1}{N\sqrt{\Phi_a(\zeta)}} < \frac{\delta}{2}$ 
for all $i\ge i'(r,\delta)$. 
Then Proposition \ref{strategy} gives the assertion. 
\end{proof}

\section{Convergence}\label{sec conv}
In this section we consider the convergence of 
$\{(X,a_i g_\Lambda)\}_i$, where $\Lambda$ is the one defined in 
Section \ref{sec const}, and $\{ a_i\}_i$ is a sequence with $a_i>0$ 
and $\lim_{i\to\infty} a_i$, 
applying Theorem \ref{key}. 
To apply them, 
we have to estimate constants $\varepsilon,C_0,C_1$ 
in {\bf (A3-6)} uniformly with respect to $i\in\N$,  and show that 
$\varepsilon \to 0$ as $i\to \infty$. 
In Section \ref{sec conv1}, we consider the uniform estimate 
for the case of $P=1$, which is the simplest case. 
In Sections \ref{sec conv2} and \ref{sec conv3}, 
we suppose $P$ is depending on some parameters. 
Then we apply them to show Theorems \ref{main result} and 
\ref{main result2} in Sections \ref{sec ex1} and \ref{sec ex2}.

Put $S_{a,n}:=a^\frac{1}{1+\alpha}K_{2n}$ and 
$T_{a,n}:=a^\frac{1}{1+\alpha}K_{2n+1}$. 
We take a subsequence 
\begin{align*}
\{ K_{n_0}<K_{n_1}<K_{n_2}<\cdots \} \subset \{ K_0<K_1<K_2<\cdots\}.
\end{align*}
We are now going to consider the convergence 
in several cases according to the 
rate of the convergence of $\{ a_i\}_i$, or the divergence of 
$\{ K_n\}_n$. 

From now on, we put 
\begin{align*}
\Phi_S^T(\zeta) &:= \Phi_{S,1}^T(\zeta) 
= \int_S^T \frac{dx}{|\zeta - (x^\alpha,0,0) |}, \\
A_S^T &:= A_{S,1}^T = \int_S^T \frac{dx}{1+x^\alpha}.
\end{align*}

\subsection{Convergence (1)}\label{sec conv1}
Fix $a>0$, $n$, $0\le S< T\le\infty$ and put $P=1$. 

\begin{prop}\label{conv8-1}
Let $R\ge 1$ and $D\le 1$. 
There exists a constant $C_\alpha>0$ depending only on $\alpha$ 
such that 
\begin{align*}
\Big| \Phi_{a} ( \zeta ) 
- \Phi_{S}^{T} (\zeta)\Big| 
&\le \frac{C_\alpha \varepsilon_n}{D},\\
\varepsilon_{a,n} 
&= a^\frac{1}{1+\alpha} + \frac{K_{2n-1}}{K_{2n}}S_{a,n}
+ (\frac{K_{2n+2}}{K_{2n+1}}T_{a,n})^{-\alpha +1}\\
&\quad\quad + |S_{a,n} - S| 
+ |T_{a,n}^{-\alpha + 1} - T^{-\alpha + 1}| 
\end{align*}
holds for any $\zeta\in K(R,D)$.
\end{prop}

\begin{proof}
By combining Propositions \ref{conv1} and \eqref{est4}\eqref{est5}, 
we have 
\begin{align*}
\Big| \Phi_{a} ( \zeta ) 
- \Phi_{S_{a,n}}^{T_{a,n}} (\zeta)\Big| 
\le \frac{2 a^\frac{1}{1+\alpha}}{D}
+ \frac{T_{a,n-1}}{D} + \frac{2S_{a,n+1}^{-\alpha +1}}{\alpha-1}
\end{align*}
if $S_{a,n+1}\ge (2|\zeta|)^\frac{1}{\alpha}$. 
Here, $|\zeta|\le R$ and 
\begin{align*}
S_{a,n+1} &= a^\frac{1}{1+\alpha}K_{2n+2} 
= \frac{K_{2n+2}}{K_{2n+1}}T_{a,n},\\
T_{a,n-1} &= a^\frac{1}{1+\alpha}K_{2n-1} 
= \frac{K_{2n-1}}{K_{2n}}S_{a,n}.
\end{align*}
On the other hand, we can see 
\begin{align*}
\Big| \Phi_{S_{a,n}}^{T_{a,n}} (\zeta) -\Phi_S^T (\zeta)\Big| 
&\le \frac{|S_{a,n} - S|}{D} 
+ \frac{2|T_{a,n}^{-\alpha + 1} - T^{-\alpha + 1}|}{\alpha-1}, 
\end{align*}
we obtain the assertion. 
\end{proof}

Now, we put $\Phi=\Phi_a,\Phi_\infty=\Phi_S^T$, 
and suppose $a$, $|S_{a,n} - S|$ and 
$|T_{a,n}^{-\alpha + 1} - T^{-\alpha + 1}|$ are sufficiently small. 
Then the constants in {\bf (A3-6)} can be taken uniformly such as
\begin{align*}
C_0 = \frac{1}{2}A_S^T,\quad 
C_1 = \frac{\alpha 2^\frac{1}{\alpha}}{\alpha - 1}, \quad 
m=1,\quad \kappa = \frac{1}{\alpha}.
\end{align*}
Then by Proposition \ref{conv8-1}, 
if $\lim_{n\to \infty}\frac{K_{2n+1}}{2n} = \infty$,
then we have $\varepsilon_{a,n}\to 0$ as $a\to 0$, $n\to \infty$, 
$|S_{a,n} - S|\to 0$ and 
$|T_{a,n}^{-\alpha + 1} - T^{-\alpha + 1}| \to 0$. 
Hence by Theorem \ref{key}, 
we have the next results.

\begin{thm}\label{main1}
Let $(X,g_\Lambda)$ be as in Section \ref{sec const} and suppose 
$\lim_{n\to \infty}\frac{K_{2n+1}}{K_{2n}} = \infty$. 
Assume that $\{ a_i\}_i\subset \R^+$ and 
\begin{align*}
\{ K_{n_0}<K_{n_1}<K_{n_2}<\cdots \} \subset \{ K_0<K_1<K_2<\cdots\}.
\end{align*}
satisfies 
\begin{align*}
\lim_{i\to \infty }a_i^\frac{1}{1+\alpha}K_{2n_i} = S\ge 0,
\quad \lim_{i\to \infty }a_i^\frac{1}{1+\alpha}K_{2n_i+1} = T\le \infty, 
\quad S<T. 
\end{align*}
Then $\{(X,a_ng_\Lambda,p)\}_n\xrightarrow{GH}(\R^3,d_S^T,0)$, 
where $d_S^T$ is the metric induced by 
$\Phi_S^T\cdot h_0$. 
\end{thm}

\subsection{Convergence (2)}\label{sec conv2}
Let $(X,d_X,p)$ and $(Y,d_Y,q)$ be pointed metric spaces 
and $\lim_{n\to \infty} a_n = 0$. 
Assume that $\{ (X,a_nd_X,p)\}_n\xrightarrow{GH}(Y,d_Y,q)$. 
Then it is easy to check that 
$\{ (X,sa_nd_X,p)\}_n\xrightarrow{GH}(Y,sd_Y,q)$ 
for any $s>0$. 
Moreover, if $\{ a_{n,m}\}_{n,m\in\N}$ satisfies $\lim_{n\to \infty}a_{n,m} =0$ 
for every $m$ and 
\begin{align*}
\{ (X,a_{n,m}d_X,p)\}_n \xrightarrow{GH} (Y_m,d_{Y_m},q_m),\quad 
\{ (Y_m,d_{Y_m},q_m)\}_m \xrightarrow{GH} (Y,d_Y,q),
\end{align*}
hold for every $m$, then by the diagonal argument 
one can show there exists a subset 
$\{ a_{n,m(n)}\}_n\subset \{ a_{n,m}\}_{n,m}$ such that 
$\lim_{n\to \infty}a_{n,m(n)} =0$ and
\begin{align*}
\{ (X,a_{n,m(n)}d_X,p)\}_n \xrightarrow{GH} (Y,d_Y,q). 
\end{align*}

Now, let $\mathcal{T}(X,d)$ be the set of isometry classes of 
tangent cones at infinity of $(X,d)$. 
From the above argument, one can see that 
$\mathcal{T}(X,d)$ is closed with respect to the pointed Gromov-Hausdorff 
topology, 
and if $(Y,d')\in\mathcal{T}(X,d)$, then its rescaling $(Y,ad')$ 
is also contained in $\mathcal{T}(X,d)$.

From Section \ref{sec conv1}, 
$(\R^3,d_S^T,0)$ may appear as the tangent cone at infinity 
of some $(X,g_\Lambda)$, where $\Lambda$ is 
as in Section \ref{sec const}.

Let $\sigma>0$, $0\le S < T\le \infty$ 
and $I_\sigma:\zeta\mapsto \sigma^{-1}\zeta$ be the dilation. 
Then we have 
\begin{align*}
I_{P^\frac{1}{1+\alpha}}^*(\Phi_S^T h_0) 
= P^\frac{-1}{\alpha}
\Phi_{P^\frac{1}{\alpha(1+\alpha)} S }
^{P^\frac{1}{\alpha(1+\alpha)} T }(\zeta) h_0
= \Phi_{S', P}^{ T' }(\zeta) h_0,  
\end{align*}
where 
\begin{align}
S' = P^\frac{-1}{1+\alpha} S,\quad 
T' = P^\frac{-1}{1+\alpha} T.\label{S'T'}
\end{align}

Hence if $(\R^3,d_S^T,0) \in \mathcal{T}(X,g_\Lambda)$,  then 
$\{(\R^3,d_{\sigma S}^{\sigma T},0)\}_{\sigma\in\R^+}$ is also 
contained in $\mathcal{T}(X,g_\Lambda)$.

\paragraph{(1).}

Fix a constant $\theta>0$, put 
$P^\frac{1}{1+\alpha} = \theta \sqrt{S^{-\alpha + 1}- T^{-\alpha + 1}}>0$, 
and let $S',T'$ be defined by \eqref{S'T'}.

\begin{prop}\label{a3forPsi}
Let $R\ge 1$.  
There exists a constant $C>0$ depending only on $\alpha$ 
such that
\begin{align*}
\Big| \Phi_{S', P}^{ T' }(\zeta) - \frac{1}{\theta^2 (\alpha - 1)}\Big| \le 
\frac{C R}{\theta^3 S^\alpha\sqrt{S^{-\alpha+1} - T^{-\alpha+1}}}
\end{align*}
holds for any $\zeta\in K(R,D)$ 
if $\theta S^\alpha\sqrt{S^{-\alpha+1}-T^{-\alpha+1}}\ge 2 R$.
\end{prop}
\begin{proof}
Note that 
\begin{align*}
\Phi_{S', P}^{ T' } (\zeta) 
= P^\frac{-1}{\alpha}\int_{P^\frac{1}{\alpha(1+\alpha)} S }
^{P^\frac{1}{\alpha(1+\alpha)} T } \frac{dx}{|\zeta-(x^\alpha,0,0)|}. 
\end{align*}
By the assumption, we have 
$P^\frac{1}{1+\alpha} S^\alpha 
= \theta S^\alpha\sqrt{S^{-\alpha+1}-T^{-\alpha+1}} 
\ge 2R$, 
then we can see 
\begin{align*}
&\quad \quad \Bigg| \int_{P^\frac{1}{\alpha(1+\alpha)} S }^{ P^\frac{1}{\alpha(1+\alpha)} T }
\frac{dx}{|\zeta - (x^\alpha,0,0)|} 
- \int_{P^\frac{1}{\alpha(1+\alpha)} S }^{ P^\frac{1}{\alpha(1+\alpha)} T }
\frac{dx}{x^\alpha} \Bigg|\\
&\le \int_{P^\frac{1}{\alpha(1+\alpha)} S }^{ P^\frac{1}{\alpha(1+\alpha)} T } 
\Big| \frac{1}{|\zeta - (x^\alpha,0,0)|} - \frac{1}{x^\alpha} \Big| dx \\
&\le \int_{P^\frac{1}{\alpha(1+\alpha)} S }^{ P^\frac{1}{\alpha(1+\alpha)} T } 
\frac{2x^\alpha |\zeta| + |\zeta|^2}
{|\zeta - (x^\alpha,0,0)|x^\alpha (|\zeta - (x^\alpha,0,0)| + x^\alpha)} dx \\
&\le \int_{P^\frac{1}{\alpha(1+\alpha)} S }^{ P^\frac{1}{\alpha(1+\alpha)} T } 
\frac{8 R}{x^{2\alpha}} dx 
+ \int_{P^\frac{1}{\alpha(1+\alpha)} S }^{ P^\frac{1}{\alpha(1+\alpha)} T } 
\frac{4 R^2}{x^{3\alpha}} dx\\
&\le \frac{8R P^\frac{-2\alpha+1}{\alpha(1+\alpha)}}{2\alpha -1} 
(S^{-2\alpha+1} - T^{-2\alpha+1})
+ \frac{4R^2 P^\frac{-3\alpha+1}{\alpha(1+\alpha)}}{3\alpha -1} (S^{-3\alpha+1} - T^{-3\alpha+1}).
\end{align*}
Since we have 
\begin{align*}
\int_{P^\frac{1}{\alpha(1+\alpha)} S }^{ P^\frac{1}{\alpha(1+\alpha)} T }  
\frac{dx}{x^\alpha} 
= \frac{P^\frac{-\alpha+1}{\alpha(1+\alpha)}}{\alpha -1}(S^{-\alpha+1} - T^{-\alpha+1})
= \frac{P^\frac{1}{\alpha}}{\theta^2 (\alpha - 1)} ,
\end{align*} 
we obtain 
\begin{align*}
\Big| \Phi_{S', P}^{ T' }(\zeta) - \frac{1}{\theta^2(\alpha -1)}\Big|
&\le \frac{8R P^\frac{-3}{1+\alpha}}{2\alpha -1} (S^{-2\alpha+1} - T^{-2\alpha+1}) \\
&\quad + \frac{4R^2 P^\frac{-4}{1+\alpha}}{3\alpha -1} (S^{-3\alpha+1} - T^{-3\alpha+1}). 
\end{align*}
Using the assumption $2R\le P^\frac{1}{1+\alpha} S^\alpha$ once more, 
we have 
\begin{align*}
\Big| \Phi_{S', P}^{ T' }(\zeta) - \frac{1}{\theta^2(\alpha -1)}\Big|
&\le \frac{8RP^\frac{-3}{1+\alpha}}{2\alpha -1} 
(S^{-2\alpha+1} - T^{-2\alpha+1}) \\
&\quad + \frac{2RP^\frac{-3}{1+\alpha}}{3\alpha -1} 
(S^{-2\alpha+1} - S^\alpha T^{-3\alpha+1})\\
&\le \theta^{-3}C_\alpha R S^{-\frac{1+\alpha}{2}}
\frac{1-(S/T)^{3\alpha -1}}{\{ 1-(S/T)^{\alpha-1}\}^\frac{3}{2}}.
\end{align*}
Now, put $f(x):=\frac{1-x^{3\alpha -1}}{ (1-x^{\alpha-1})^\frac{3}{2} }$ 
for $0\le x<1$. 
Then there exists a constant $C_\alpha'>0$ such that 
$f(x) \le C_\alpha'(1-x^{\alpha-1})^{-\frac{1}{2}}$ holds for all $0\le x<1$. 
Consequently, by replacing $C_\alpha$ larger if necessary, we can see 
\begin{align*}
\Big|\Phi_{S', P}^{ T' }(\zeta) - \frac{1}{\theta^2(\alpha -1)}\Big|
\le \frac{C_\alpha R}{\theta^3 S^\alpha\sqrt{S^{-\alpha+1} - T^{-\alpha+1}}}.
\end{align*}
\end{proof}

\begin{prop}
Suppose $\theta S^\alpha\sqrt{S^{-\alpha+1} - T^{-\alpha+1}} 
\ge 2R$ for $R\ge 1$. 
Then 
\begin{align*}
A_{S',P}^{T'} \ge \frac{1}{2\theta^2 (\alpha -1)},\quad 
\Phi_{S',P}^{T'} (\zeta_\R,\zeta_\C) \le 
\frac{ 2|\zeta|}{\theta^2 (\alpha -1) |\zeta_\C|}
\end{align*}
holds for any $\zeta = (\zeta_\R,\zeta_\C)\in \R^3=\R\oplus\C$ 
with $|\zeta|\le R$. 
\label{C_0 C_1'}
\end{prop}
\begin{proof}
We have 
$1\le S^{-\alpha} x^\alpha$ for all $x\ge S$, then we can see 
\begin{align*}
A_{S',P}^{T'} 
&\ge 
P^{-\frac{1}{1+\alpha}} \int_S^T 
\frac{dx}{S^{-\alpha}x^\alpha + P^\frac{1}{1+\alpha}x^\alpha} \\
&= 
\frac{1}{P^\frac{1}{1+\alpha}(S^{-\alpha}+P^\frac{1}{1+\alpha})} 
\int_S^T \frac{dx}{x^\alpha} \\
&= 
\frac{1}{P^\frac{1}{1+\alpha}S^{-\alpha}(1+S^\alpha P^\frac{1}{1+\alpha})} 
\frac{S^{-\alpha+1} - T^{-\alpha+1}}{\alpha -1}.
\end{align*}
Since we have 
\begin{align*}
S^\alpha P^\frac{1}{1+\alpha}
= \theta \sqrt{S^{-\alpha +1} - T^{-\alpha +1} } 
\ge 2R \ge 1,  
\end{align*}
we obtain 
\begin{align*}
A_{S',P}^{T'} 
\ge \frac{S^{-\alpha+1} - T^{-\alpha+1}}
{2(\alpha -1) P^\frac{1}{1+\alpha}S^{-\alpha}\cdot S^\alpha P^\frac{1}{1+\alpha}}
= \frac{1}{2\theta^2(\alpha -1)}
\end{align*}

Next we consider the upper estimate of $\Phi_{S',P}^{T'} (\zeta)$. 
Take $\zeta$ such that $|\zeta|\le R$, then 
we have $2|\zeta| \le P^\frac{1}{1+\alpha}S^\alpha$ by the assumption. 
Then one can see 
\begin{align*}
\Phi_{S',P}^{T'} (\zeta) 
\le P^{-\frac{1}{1+\alpha}} 
\int_S^T \frac{2 dx}{P^\frac{1}{1+\alpha}x^\alpha} 
&=
P^{\frac{-2}{1+\alpha}} 
\frac{2(S^{-\alpha+1} - T^{-\alpha+1})}{\alpha-1} \\
&\le \frac{ 2}{\theta^2(\alpha-1)}
\le \frac{ 2 |\zeta|}{\theta^2(\alpha-1)|\zeta_\C|}.
\end{align*}
\end{proof}

\begin{prop}\label{prop Psi^S}
Let $\Phi=\Phi_{S',P}^{T'}$ and 
$\Phi_\infty \equiv \frac{1}{\theta^2(\alpha-1)}$. 
Then there exists $C>0$ such that 
$\Phi,\Phi_\infty$ satisfy {\bf (A3-6)} for $R\ge 1$ and  
\begin{align*}
&m=1,\quad 
\varepsilon = \frac{C R}{\theta^3 S^\alpha\sqrt{S^{-\alpha+1} - T^{-\alpha+1}}}, \quad 
C_0 = \frac{1}{2\theta^2 (\alpha-1)}, \\
&C_1 = \frac{1}{\theta^2} 
\max\Big\{ \frac{1}{\alpha-1}, \frac{C}{2}\Big\}, \quad 
\kappa = 1,
\end{align*}
if $\theta S^\alpha\sqrt{S^{-\alpha+1}-T^{-\alpha+1}} \ge 2 R$.
\end{prop}
\begin{proof}
It is obvious that {\bf (A4)} holds. 
Proposition \ref{C_0 C_1'} gives {\bf (A5)} for 
$C_0=\frac{1}{2\theta^2(\alpha -1)}$ 
if we take 
$\theta S^\alpha\sqrt{S^{-\alpha+1}-T^{-\alpha+1}} \ge 2R$. 
{\bf (A6)} holds for $C_1 =\frac{1}{\theta^2(\alpha-1)}$ since 
$\frac{1}{\alpha-1} = \frac{1}{\alpha-1}\frac{|\zeta|}{|\zeta|}
\le \frac{1}{\alpha-1}\frac{|\zeta|}{|\zeta_\C|}$. 
Combining $\theta S^\alpha\sqrt{S^{-\alpha+1}-T^{-\alpha+1}} \ge 2R$ and 
Proposition \ref{a3forPsi}, we can see 
$\varepsilon\le \frac{C}{2\theta^2}$.
\end{proof}

Now, Propositions \ref{prop isom} and 
\ref{prop Psi^S} with $\theta=1$ gives the following theorem.

\begin{thm}\label{limit1}
Let $\{ S_i\}_i$ and $\{ T_i\}_i$ be sequences such that 
$0\le S_i< T_i\le \infty$ and 
$\lim_{i\to \infty} S_i^\alpha\sqrt{S_i^{-\alpha+1}-T_i^{-\alpha+1}} = \infty$, 
then $\{ (\R^3, d_{S_i}^{T_i},0) \}_i$ converges to 
$(\R^3, h_0,0)$ in the pointed Gromov-Hausdorff topology. 
\end{thm}

\paragraph{(2).}

Next put $P^\frac{1}{1+\alpha} = \theta|T-S|$ for 
$0\le S<T$ and $\theta>0$, 
and let $S',T'$ be as in \eqref{S'T'}. 
Then we can show the following similarly to Proposition \ref{prop Psi^T}.

\begin{prop}\label{a3forPsi2}
Let $D\ge 1$. 
We have 
\begin{align*}
\Big| \Phi_{S',P}^{T'}(\zeta) - \frac{1}{\theta |\zeta|} \Big| 
&\le \frac{2}{\theta D}, \\
\Big| \Phi_{S',P}^{T'}(\zeta) - \frac{1}{\theta |\zeta|} \Big| &\le 
\frac{1+\theta T^\alpha (T-S)}{D^3} 
T^{\alpha}(T-S). 
\end{align*}
for all $\zeta\in K(R,D)$.
\end{prop}
\begin{proof}
The first inequality is obviously shown by 
$\Phi_{S',P}^{T'}(\zeta) \le \frac{1}{\theta D}$ and 
$\frac{1}{|\zeta|} \le \frac{1}{D}$. 
The second inequality follows from 
\begin{align*}
&\quad \quad \Big| \int_{P^\frac{1}{\alpha(1+\alpha)} S }^{ P^\frac{1}{\alpha(1+\alpha)} T } \frac{dx}{|\zeta - (x^\alpha,0,0)|}
- \int_{P^\frac{1}{\alpha(1+\alpha)} S }^{ P^\frac{1}{\alpha(1+\alpha)} T } 
\frac{dx}{|\zeta|} \Big| \\
&\le \int_{P^\frac{1}{\alpha(1+\alpha)} S }^{ P^\frac{1}{\alpha(1+\alpha)} T } 
\frac{2 x^\alpha |\zeta| + x^{2\alpha}}
{|\zeta - (x^\alpha,0,0)||\zeta|(|\zeta - (x^\alpha,0,0)| + |\zeta|)}dx \\
&\le \int_{P^\frac{1}{\alpha(1+\alpha)} S }^{ P^\frac{1}{\alpha(1+\alpha)} T } 
\frac{2 x^\alpha}{D^2} dx
+ \int_{P^\frac{1}{\alpha(1+\alpha)} S }^{ P^\frac{1}{\alpha(1+\alpha)} T } 
\frac{x^{2\alpha}}{D^3} dx \\
&\le C_\alpha \frac{ P^\frac{1}{\alpha} (T^{\alpha+1}-S^{\alpha+1}) 
+ P^\frac{2\alpha + 1}{\alpha(1+\alpha)} (T^{2\alpha+1}-S^{2\alpha+1})}{D^3} \\
&= C_\alpha \theta^{1+\frac{1}{\alpha}}T^{\alpha+1}(T-S)^{1+\frac{1}{\alpha}} 
\frac{ 1-(S/T)^{\alpha+1} 
+ \theta (T-S) T^{\alpha}(1-(S/T)^{2\alpha+1})}{D^3}
\end{align*}
By the similar argument to Proposition \ref{a3forPsi}, 
we can replace $1-(S/T)^{\alpha+1}$ or $1-(S/T)^{2\alpha+1}$ by $1-S/T$, hence 
we obtain the assertion.
\end{proof}

\begin{prop}
\begin{align*}
A_{S',P}^{T'} \ge \frac{1}{\theta (1+\theta T^\alpha( T-S ))},\quad 
\Phi_{S',P}^{T'} (\zeta_\R,\zeta_\C) \le 
\frac{1}{\theta|\zeta_\C|}
\end{align*}
holds for any $\zeta = (\zeta_\R,\zeta_\C)\in \R^3=\R\oplus\C$. 
\label{C_0 C_1}
\end{prop}
\begin{proof}
One can see 
\begin{align*}
A_{S',P}^{T'}
= P^{-\frac{1}{1+\alpha}} \int_S^T
\frac{dx}{1+P^\frac{1}{1+\alpha}x^\alpha}
&\ge P^{-\frac{1}{1+\alpha}} \int_S^T 
\frac{dx}{1+P^\frac{1}{1+\alpha} T^\alpha }\\
&\ge \frac{T - S}{P^{\frac{1}{1+\alpha}}
(1 + P^\frac{1}{1+\alpha}T^\alpha )}\\
&= \frac{1}{\theta( 1 + \theta T^\alpha (T-S) )}. 
\end{align*}

We can also obtain 
\begin{align*}
\Phi_{S',P}^{T'} (\zeta) 
\le \frac{T-S}{P^{\frac{1}{1+\alpha}}|\zeta_\C|}
= \frac{1}{\theta |\zeta_\C|}.
\end{align*}
\end{proof}

Combining Propositions \ref{a3forPsi2} and \ref{C_0 C_1}, 
the next proposition is obtained.

\begin{prop}\label{prop Psi^T}
Let $\Phi=\Phi_{S',P}^{T'}$ and $\Phi_\infty(\zeta) = \frac{1}{\theta|\zeta|}$. 
Then 
$\Phi,\Phi_\infty$ satisfy {\bf (A3-6)} for $R\ge 1$ and  
\begin{align*}
m&=3,\quad 
\varepsilon = (1+\theta T^\alpha (T-S)) T^{\alpha}(T-S), \\
C_0 &= \frac{1}{\theta( 1+\theta T^\alpha (T-S))}, \quad 
C_1 = \frac{2}{\theta}, \quad 
\kappa = 0,
\end{align*}
for any $0\le S<T$. 
\end{prop}

By Propositions \ref{prop isom} and \ref{prop Psi^T} for $\theta=1$, 
we have the next result. 
\begin{thm}\label{limit2}
Let $\{ S_i\}_i$ and $\{ T_i\}_i$ be a sequence 
such that $0\le S_i<T_i$ and 
$\lim_{i\to \infty} T_i^\alpha|T_i-S_i|=0$, 
then $\{ (\R^3, d_{S_i}^{T_i},0) \}_i$ converges to 
$(\R^3, \frac{1}{|\zeta|}h_0,0)$ in the pointed Gromov-Hausdorff topology. 
\end{thm}

\subsection{Convergence (3)}\label{sec conv3}
Here, we fix $a>0$ and $n$ and suppose that 
$T_{a,n}=a^\frac{1}{1+\alpha}K_{2n+1}$ is sufficiently small 
and $S_{a,n+1}=a^\frac{1}{1+\alpha}K_{2n+2}$ is sufficiently large. 
Fix $P$ and $\theta$ such that 
\begin{align*}
P^\frac{1}{1+\alpha} = \theta( T_{a,n} - S_{a,n} ) 
= \sqrt{S_{a,n+1}^{-\alpha + 1} - T_{a,n+1}^{-\alpha + 1}}. 
\end{align*}
Put $S'_l=P^\frac{-1}{1+\alpha}S_{a,l}$ and $T'_l=P^\frac{-1}{1+\alpha}T_{a,l}$.

\begin{prop}\label{conv conf flat}
Let $R\ge 1$ and $D\le 1$,  and $P$ be as above. 
Assume $P^\frac{1}{\alpha(1+\alpha)}S_{a,n+2}\ge (2R)^\frac{1}{\alpha}$. 
Then there exists a constant $C_\alpha>0$ depending only on 
$\alpha$ such that 
\begin{align*}
|\Phi_a(\zeta) - \Phi_{S'_n,P}^{T'_n} (\zeta)
-  \Phi_{S'_{n+1},P}^{T'_{n+1}}(\zeta)| 
\le \frac{C_\alpha\varepsilon_{a,n}}{D},
\end{align*}
for any $\zeta\in K(R,D)$, 
where $\varepsilon_{a,n}$ is the constant defined by 
\begin{align*}
\varepsilon_{a,n} 
= \frac{1+K_{2n-1}}{\theta( K_{2n+1} - K_{2n})} 
+ \frac{K_{2n+4}^{-\alpha+1}}{K_{2n+2}^{-\alpha+1}-K_{2n+3}^{-\alpha+1}}.
\end{align*}
\end{prop}

\begin{proof}
By Propositions \ref{conv1} and \eqref{est4}\eqref{est5}, 
we have 
\begin{align*}
|\Phi_a - \Phi_{S'_n,P}^{T'_n} 
-  \Phi_{S'_{n+1},P}^{T'_{n+1}}| 
\le \frac{2(\frac{a}{P})^\frac{1}{1+\alpha} +P^\frac{-1}{1+\alpha}T_{a,n-1}}{D} 
+ \frac{2S_{a,n+2}^{-\alpha + 1}}{P^\frac{2}{1+\alpha}(\alpha -1)},
\end{align*}
if $P^\frac{1}{\alpha(1+\alpha)}S_{a,n+2}\ge (2R)^\frac{1}{\alpha}$. 
Since we have 
\begin{align*}
\Big(\frac{a}{P}\Big)^\frac{1}{1+\alpha} 
&= \frac{1}{\theta( K_{2n+1} - K_{2n})},\\ 
P^\frac{-1}{1+\alpha}T_{a,n-1} 
&= \frac{K_{2n-1}}{\theta( K_{2n+1} - K_{2n})},\\
\frac{S_{a,n+2}^{-\alpha + 1}}{P^\frac{2}{1+\alpha}}
&= \frac{K_{2n+4}^{-\alpha+1}}
{K_{2n+2}^{-\alpha+1}-K_{2n+3}^{-\alpha+1}},
\end{align*}
then we have the assertion. 
\end{proof}

Here, the assumption $P^\frac{1}{\alpha(1+\alpha)}S_{a,n+2}\ge 
(2R)^\frac{1}{\alpha}$ can be 
replaced by 
\begin{align*}
\Big( \frac{K_{2n+4}}{K_{2n+2}}\Big)^\alpha
S_{a,n+1}^\alpha \sqrt{ S_{a,n+1}^{-\alpha + 1} - T_{a,n+1}^{-\alpha + 1} }
\ge 2R.
\end{align*}

We can apply Propositions \ref{a3forPsi} and \ref{a3forPsi2} to 
$\Phi_{S'_n,P}^{T'_n}$ and $\Phi_{S'_{n+1},P}^{T'_{n+1}}$. 
If we put 
\begin{align*}
S=S_{a,n+1},\quad T=T_{a,n+1},\quad \theta=1,\quad 
P^\frac{1}{1+\alpha} 
= \sqrt{S_{a,n+1}^{-\alpha +1} - T_{a,n+1}^{-\alpha +1}}, 
\end{align*}
in Proposition \ref{a3forPsi}, then we have 
\begin{align*}
\Big|\Phi_{S'_{n+1},P}^{T'_{n+1}} - \frac{1}{\alpha - 1}\Big| \le 
\frac{C R}{S_{a,n+1}^\alpha\sqrt{S_{a,n+1}^{-\alpha + 1}
- T_{a,n+1}^{-\alpha + 1}}}
\end{align*}
for any $\zeta\in K(R,D)$ 
if $S_{a,n+1}^\alpha\sqrt{S_{a,n+1}^{-\alpha + 1} 
- T_{a,n+1}^{-\alpha + 1}}\ge 2 R$.

If we put 
\begin{align*}
S=S_{a,n},\quad 
T=T_{a,n},\quad 
P^\frac{1}{1+\alpha} = \theta (T_{a,n} -S_{a,n}),
\end{align*}
in Proposition \ref{a3forPsi2}, then we have 
\begin{align*}
\Big|\Phi_{S'_n,P}^{T'_n} - \frac{1}{\theta|\zeta|}\Big| 
&\le \frac{2}{\theta D},\\ 
\Big|\Phi_{S'_n,P}^{T'_n} - \frac{1}{\theta|\zeta|}\Big| 
&\le \frac{ 1+\theta T_{a,n}^\alpha (T_{a,n} -S_{a,n}) }{D^3} 
T_{a,n}^\alpha (T_{a,n} -S_{a,n}).
\end{align*}

Now, we put $\Phi=\Phi_a$, 
$\Phi_\infty = \frac{1}{\alpha-1} + \frac{1}{\theta|\zeta|}$. 
Combining above arguments and Proposition \ref{conv conf flat}, 
we can describe $\varepsilon,C_1$ in {\bf (A3)} 
explicitly, with $m=3$. 
Moreover, by Propositions \ref{a3forPsi}, \ref{a3forPsi2}, 
\ref{C_0 C_1'} and \ref{C_0 C_1}, we obtain 
$C_0,C_1$ in {\bf (A5-6)} and $\kappa=1$. 
Fix a constant $A>0$ and suppose  
\begin{align*}
A^{-1} \le \theta \le A,\quad 
S_{a,n+1}^\alpha\sqrt{S_{a,n+1}^{-\alpha + 1} 
- T_{a,n+1}^{-\alpha + 1}}\ge 2 R,
\end{align*}
and  and $P$ is as above. 
Then we can take these constants in {\bf (A3-6)} 
being only depending on $\alpha,A,R$, 
if $\varepsilon_{a,n}$, $S_{a,n+1}^{-\alpha}
( S_{a,n+1}^{-\alpha+1} - T_{a,n+1}^{-\alpha+1} )^\frac{-1}{2}$ 
and $T_{a,n}^\alpha( T_{a,n} - S_{a,n} )$ are sufficiently small. 
Therefore, we obtain the following result.

\begin{thm}\label{main2}
Let $(X,g_\Lambda)$ be as in Section \ref{sec const}, 
take a subsequence 
\begin{align*}
\{ K_{n_0}<K_{n_1}<K_{n_2}<\cdots \} \subset \{ K_0<K_1<K_2<\cdots\},
\end{align*}
and suppose 
\begin{align}
\lim_{i\to \infty } \Big\{\frac{ K_{2n_i-1} }{ K_{2n_i+1} - K_{2n_i}} 
+ \frac{K_{2n_i+4}^{-\alpha+1}}{K_{2n_i+2}^{-\alpha+1}-K_{2n_i+3}^{-\alpha+1}} 
\Big\}
= 0.\label{condi K-1}
\end{align}
If a sequence $\{ a_i\}_i\subset \R^+$ satisfies 
\begin{align*}
\lim_{i\to \infty} \frac{ \sqrt{S_{a_i,n_i+1}^{-\alpha + 1}
- T_{a_i,n_i+1}^{-\alpha + 1}}}{ T_{a_i,n_i} - S_{a_i,n_i} } 
&= \theta >0,\\ 
\lim_{i\to \infty} S_{a_i,n_i+1}^{-\alpha} 
( S_{a_i,n_i+1}^{-\alpha+1} - T_{a_i,n_i+1}^{-\alpha+1} )^\frac{-1}{2} 
&= \lim_{i\to \infty} T_{a_i,n_i}^\alpha (T_{a_i,n_i} -S_{a_i,n_i}) = 0, 
\end{align*}
then $\{(X,a_ig_\Lambda,p)\}_n\xrightarrow{GH}
(\R^3,(\frac{1}{\alpha-1} + \frac{1}{\theta|\zeta|})h_0,0)$. 
\end{thm}

Next we estimate $\Phi_a - \frac{1}{\alpha - 1}$ in the same situation, 
as $\theta\to\infty$. 
We have 
\begin{align*}
|\Phi_a -  \Phi_{S'_{n+1},P}^{T'_{n+1}}| 
&\le \frac{2(\frac{a}{P})^\frac{1}{1+\alpha} +P^\frac{-1}{1+\alpha}T_{a,n-1} 
+ P^\frac{-1}{1+\alpha}(T_{a,n} - S_{a,n})}{D} 
+ \frac{2S_{a,n+2}^{-\alpha + 1}}{P^\frac{2}{1+\alpha}(\alpha -1)}\\ 
&\le \frac{C_\alpha}{D}
\Big\{
\frac{1+K_{2n-1}}{\theta( K_{2n+1} - K_{2n})} + \frac{1}{\theta}
+ \frac{K_{2n+4}^{-\alpha+1}}
{K_{2n+2}^{-\alpha+1}-K_{2n+3}^{-\alpha+1}}
\Big\}. 
\end{align*}
Applying Propositions \ref{a3forPsi} and \ref{C_0 C_1'} 
with $\theta = 1$ and \eqref{est2}, we have 
\begin{align*}
\Big|\Phi_a -  \frac{1}{\alpha - 1}\Big| 
&\le \frac{C_\alpha}{D}
\Bigg\{
\frac{1+K_{2n-1}}{\theta( K_{2n+1} - K_{2n})} + \frac{1}{\theta}
+ \frac{K_{2n+4}^{-\alpha+1}}
{K_{2n+2}^{-\alpha+1}-K_{2n+3}^{-\alpha+1}}\\
&\quad \quad \quad \quad + \frac{R}{S_{a,n+1}^\alpha 
\sqrt{S_{a,n+1}^{-\alpha+1}-T_{a,n+1}^{-\alpha+1}}}\Bigg\},\\
\Phi_a &\ge \Big( A_{S'_{n+1},P}^{T'_{n+1}} 
- \frac{2}{\theta( K_{2n+1} - K_{2n})}\Big)\min\Big\{ \frac{1}{|\zeta|},1\Big\}\\
&\ge \Big( \frac{1}{2(\alpha -1)}
- \frac{2}{\theta( K_{2n+1} - K_{2n})}\Big)\min\Big\{ \frac{1}{|\zeta|},1\Big\}\\
\end{align*}
if $D\le 1$, $R\ge 1$ and $|\zeta|\le R$.  
Therefore, we can take 
$C_0,C_1,\kappa,m$ in {\bf (A3-6)} depending only on 
$\alpha, R$ if $\varepsilon \to 0$, where $\Phi=\Phi_a, \Phi_\infty =\frac{1}{\alpha-1}$,
hence we have the following theorem.

\begin{thm}\label{main3}
Let $(X,g_\Lambda)$ be as in Section \ref{sec const} and  
suppose $\{ K_{n_i}\}_i$ satisfies \eqref{condi K-1}. 
If a sequence $\{ a_i\}_i\subset \R^+$ satisfies 
\begin{align*}
\lim_{i\to \infty} \frac{ \sqrt{S_{a_i,n_i+1}^{-\alpha + 1}
- T_{a_i,n_i+1}^{-\alpha + 1}} }{ T_{a_i,n_i} - S_{a_i,n_i} } 
&= \infty,\\ 
\lim_{i\to \infty} S_{a_i,n_i+1}^{-\alpha} 
( S_{a_i,n_i+1}^{-\alpha+1} - T_{a_i,n_i+1}^{-\alpha+1} )^\frac{-1}{2}
&= 0, 
\end{align*}
then $\{(X,a_ig_\Lambda,p)\}_n\xrightarrow{GH}
(\R^3,h_0,0)$. 
\end{thm}

By the similar argument, we have the following. 

\begin{thm}\label{main4}
Let $(X,g_\Lambda)$ be as in Section \ref{sec const} and  
suppose $\{ K_{n_i}\}_i$ satisfies \eqref{condi K-1}. 
If a sequence $\{ a_i\}_i\subset \R^+$ satisfies 
\begin{align*}
\lim_{i\to \infty} \frac{ \sqrt{S_{a_i,n_i+1}^{-\alpha + 1}
- T_{a_i,n_i+1}^{-\alpha + 1}} } { T_{a_i,n_i} - S_{a_i,n_i} } 
&= 0,\\ 
\lim_{i\to \infty} S_{a_i,n_i+1}^{-\alpha} 
( S_{a_i,n_i+1}^{-\alpha+1} - T_{a_i,n_i+1}^{-\alpha+1} )^\frac{-1}{2} 
&= \lim_{i\to \infty} T_{a_i,n_i}^\alpha (T_{a_i,n_i} -S_{a_i,n_i}) = 0, 
\end{align*}
then $\{(X,a_ig_\Lambda,p)\}_n\xrightarrow{GH}
(\R^3,\frac{1}{|\zeta|}h_0,0)$. 
\end{thm}
\begin{proof}
Put 
\begin{align*}
P^\frac{1}{1+\alpha} &:= ( T_{a,n} - S_{a,n} ) 
= \theta\sqrt{S_{a,n+1}^{-\alpha + 1} - T_{a,n+1}^{-\alpha + 1}}, \\
S'_l &=P^\frac{-1}{1+\alpha} S_{a,l},\quad 
T'_l =P^\frac{-1}{1+\alpha} T_{a,l}. 
\end{align*}
The similar argument to \eqref{est4} gives
\begin{align*}
\Phi_{S'_{n+1},P}^{T'_{n+1}} (\zeta) 
\le \frac{2((S'_{n+1})^{-\alpha +1} - (T'_{n+1})^{-\alpha +1})}{P(\alpha -1)}
\end{align*}
if $P(S'_{n+1})^\alpha \ge 2R$, which is equivalent to 
$\theta S_{a,n+1}^\alpha\sqrt{S_{a,n+1}^{-\alpha+1} - T_{a,n+1}^{-\alpha+1}}
\ge 2R$, 
then the similar argument to Proposition \ref{conv conf flat} gives 
\begin{align*}
|\Phi_a(\zeta) - \Phi_{S'_n,P}^{T'_n} (\zeta)| 
\le \frac{C_\alpha\varepsilon_{a,n}}{D} 
+ \frac{2}{(\alpha-1) \theta},
\end{align*}
for any $\zeta\in K(R,D)$, 
where $\varepsilon_{a,n}$ is the constant defined by 
\begin{align*}
\varepsilon_{a,n} 
= \frac{1+K_{2n-1}}{K_{2n+1} - K_{2n}} 
+ \frac{K_{2n+4}^{-\alpha+1}}
{\theta^2 (K_{2n+2}^{-\alpha+1}-K_{2n+3}^{-\alpha+1})}.
\end{align*}
Moreover, Proposition \ref{a3forPsi2} with $\theta=1$ gives 
\begin{align*}
|\Phi_{S'_n,P}^{T'_n} (\zeta) - \frac{1}{|\zeta|}| &\le \frac{2}{D},\\
|\Phi_{S'_n,P}^{T'_n} (\zeta) - \frac{1}{|\zeta|}| 
&\le \frac{1+T_{a,n}^\alpha (T_{a,n} - S_{a,n})}{D^3}
T_{a,n}^\alpha (T_{a,n} - S_{a,n}). 
\end{align*} 
Then we can see $|\Phi_a - \frac{1}{|\zeta|}| \le \frac{\varepsilon}{D^3}$ 
for some $\varepsilon>0$ if $D\le 1$ and $\zeta\in K(R,D)$. 
Here, $\varepsilon$ goes to $0$ as 
$\varepsilon_{a,n} \to 0$, $\theta\to \infty$, 
$S_{a,n+1}^\alpha\sqrt{S_{a,n+1}^{-\alpha+1} 
- T_{a,n+1}^{-\alpha+1}}\to \infty$ 
and $T_{a,n}^\alpha (T_{a,n} - S_{a,n}) \to 0$. 
Since one can take $C_0,C_1,m,\kappa$ in {\bf (A3-6)} depending only on 
$\alpha$ if $\varepsilon$ is sufficiently small, 
by Proposition \ref{C_0 C_1} with $\theta=1$ and \eqref{est2}. 
\end{proof}

\subsection{Example (1)}\label{sec ex1}
Let $\Lambda$ be as in Section \ref{sec const}. 
Moreover we take and increasing sequence $\{ K_n\}_n$ 
such that 
\begin{align*}
\lim_{n\to \infty}\frac{K_{n}}{K_{n-1}} =\infty. 
\end{align*}
In this situation, we observe which pointed metric spaces 
can be contained in $\mathcal{T}(X, g_\Lambda)$ and 
prove Theorem \ref{main result}. 

Take $S>0$ and put 
$a_i:=K_{2i}^{-1-\alpha}S^{1+\alpha}$. 
Then we have $a_i^{\frac{1}{1+\alpha}}K_{2i} = S$ 
and $\lim_i a_i^{\frac{1}{1+\alpha}}K_{2i+1} =\infty$. 
Hence Theorem \ref{main1} implies that 
$(X,a_ig_\Lambda,p)\xrightarrow{GH} (\R^3,d_S^\infty,0)$. 
Similarly, if we take 
$a_i:=K_{2i+1}^{-1-\alpha}T^{1+\alpha}$ for $T>0$ then 
we obtain $(\R^3,d_0^T,0)$ as the pointed Gromov-Hausdorff limit. 

Next we fix $\theta>0$ and put 
$a_i=\theta^{-1}K_{2i+1}^{-2}K_{2i+2}^{-\alpha + 1}$. 
Then one can check that the assumptions of 
Theorem \ref{main2} is satisfied, hence one obtain 
$(\R^3,(\frac{1}{\alpha-1} + \frac{1}{\theta|\zeta|})h_0,0)$ 
as the pointed Gromov-Hausdorff limit. 
By taking the limit $\theta\to 0$ or $\theta\to \infty$, 
we obtain 
$(\R^3,h_0,0)$ and $(\R^3,\frac{1}{|\zeta|}h_0,0)$ 
as the pointed Gromov-Hausdorff limit.

In fact, we obtain the next result. 

\begin{thm}\label{ex1}
Let $\Lambda,\{ K_n\}_n$ satisfy 
$\lim_{n\to \infty}\frac{K_{n}}{K_{n-1}} =\infty$. 
Then $\mathcal{T}(X, g_\Lambda)$ is equal to 
the closure of 
\begin{align*}
\{ (\R^3, s d_1^\infty,0); s>0 \}\cup 
\{ (\R^3, s d_0^1,0); s>0 \} \cup 
\{ (\R^3, s \Big( 1+\frac{1}{|\zeta|} \Big) h_0,0); s>0 \}
\end{align*}
with respect to the Gromov-Hausdorff topology. 
Moreover we have 
\begin{align*}
\lim_{s\to \infty} (\R^3, s d_1^\infty,0) 
&= \lim_{s\to 0} \Big(\R^3, s \Big( 1+\frac{1}{|\zeta|} \Big) h_0,0\Big)
= (\R^3,h_0,0),\\
\lim_{s\to 0}  (\R^3, s d_0^1,0)
&= \lim_{s\to \infty } \Big(\R^3, s \Big( 1+\frac{1}{|\zeta|} \Big),0\Big) 
= \Big(\R^3,\frac{1}{|\zeta|}h_0,0\Big),\\
\lim_{s\to 0} (\R^3, s d_1^\infty,0) 
&= \lim_{s\to \infty}  (\R^3, s d_0^1,0) 
= (\R^3, d_0^\infty,0).
\end{align*}
\end{thm}
\begin{proof}
We have already shown that the pointed metric spaces in the 
above list are contained in $\mathcal{T}(X, g_\Lambda)$. 
Accordingly, what we have to show is that 
any other pointed metric spaces may not arise as 
the tangent cone at infinity of $(X,g_\Lambda)$.

Suppose that a sequence $\{ a_i\}_i \subset \R^+$ is 
given such that 
$(X,a_ig_\Lambda,p)\xrightarrow{GH} (Y,d,q)$ as $i\to \infty$. 
It suffices to show that $(Y,d,q)$ is one of the metric spaces 
in the list. 

First of all, we may assume that for any large $M>0$ 
there exists $i(M)$ such that 
\begin{align*}
\{ a_i^\frac{1}{1+\alpha}K_n \in \R^+; n\in\N\} 
\cap [M^{-1},M]
\end{align*}
is empty for any $i\ge i(M)$. 
If not, there is $M>0$ and a map $i\mapsto n_i$ such that 
$M^{-1} \le a_i^\frac{1}{1+\alpha}K_{n_i} \le M$ holds 
for infinitely many $i$. 
Then by taking subsequence $\{ a_{i_j}\} \subset \{ a_i\}_i$, 
we may suppose $M^{-1} \le a_{i_j}^\frac{1}{1+\alpha}K_{2n_{i_j}} \le M$ holds 
for any $j$ or $M^{-1} \le a_{i_j}^\frac{1}{1+\alpha}K_{2n_{i_j}+1} \le M$ holds 
for any $j$. 
If the former case holds, then by replacing by subsequence 
we may suppose  
\begin{align*}
\lim_{i\to \infty}a_i^\frac{1}{1+\alpha}K_{2n_i} 
&=S\in [M^{-1},M], \\
\lim_{i\to \infty}a_i^\frac{1}{1+\alpha}K_{2n_i+1} 
&= S\lim_{i\to \infty}\frac{K_{2n_i+1}}{K_{2n_i}}=\infty, 
\end{align*}
and we can apply Theorem \ref{main1} hence obtain 
$(Y,d,q) = (\R^3, d_S^\infty,0)$. 
If the latter case holds, then we have $(Y,d,q) = (\R^3, d_0^T,0)$ 
for some $T>0$.

Now, we may suppose that there exists $l_i\in\N$ for each $i$ 
such that $\lim_{i\to \infty}a_i^\frac{1}{1+\alpha}K_{l_i} = 0$ 
and $\lim_{i\to \infty}a_i^\frac{1}{1+\alpha}K_{l_i+1} = \infty$ hold. 
If $\{ i\in \N; l_i\ {\rm is\ even}.\}$ is an infinite set, 
then we can apply Theorem \ref{main1} again and obtain 
$(Y,d,q) = (\R^3, d_0^\infty,0)$. 
Therefore, replacing by subsequence, we may suppose 
\begin{align*}
\lim_{i\to \infty}a_i^\frac{1}{1+\alpha}K_{2n_i+1} = 0,\quad 
\lim_{i\to \infty}a_i^\frac{1}{1+\alpha}K_{2n_i+2} = \infty.
\end{align*}
Now, we have 
\begin{align*}
\sqrt{S_{a,n+1}^{-\alpha+1} - T_{a,n+1}^{-\alpha+1}} 
\ge \frac{S_{a,n+1}^{\frac{-\alpha+1}{2}}}{2},\quad 
T_{a,n}-S_{a,n} \ge \frac{T_{a,n}}{2}
\end{align*}
holds for sufficiently large $n$. 
Hence if  
\begin{align*}
0< \liminf_{i\to \infty}\frac{S_{a_i,n_i+1}^{\frac{1-\alpha}{2}}}{T_{a_i,n_i}}
\le \limsup_{i\to \infty}\frac{S_{a_i,n_i+1}^{\frac{1-\alpha}{2}}}{T_{a_i,n_i}}
<\infty
\end{align*}
holds, then Theorem \ref{main2} can be applied to this situation 
by taking a subsequence, then we obtain 
$(Y,d,q) = (\R^3, \Big( 1+\frac{\theta}{|\zeta|} \Big) h_0,0)$ 
for some $\theta>0$. 
Hence the remaining cases are 
\begin{align*}
\lim_{i\to \infty}\frac{S_{a_i,n_i+1}^{\frac{1-\alpha}{2}}}{T_{a_i,n_i}} = 0\quad
{\rm or} 
\quad \lim_{i\to \infty}\frac{S_{a_i,n_i+1}^{\frac{1-\alpha}{2}}}{T_{a_i,n_i}}
=\infty.
\end{align*}
In both of the cases, we can apply 
Theorems \ref{main3} or \ref{main4}, then obtain 
$(Y,d,q) = (\R^3, h_0,0)$ or $(\R^3, \frac{1}{|\zeta|}h_0,0)$.
\end{proof}

One can also see that there are no 
nontrivial isometries between two pointed metric spaces 
appearing in the list of Theorem \ref{ex1}. 
Here, the isometry of pointed metric spaces means the 
bijective morphism preserving the metrics and the base points. 

\begin{table}[htb]
\begin{center}
    \caption{Tangent cones ($0<S,T,\theta<\infty$)}\label{table1}
 \begin{tabular}{|c||c|c|}\hline
metric & tangent cone at $0$ & tangent cone at $\infty$ \\ \hline \hline
$d_S^T\ (S<T)$ & $h_0$ & $\frac{1}{|\zeta|}h_0$ \\ \hline
$d_S^\infty$ & $h_0$ & $d_0^\infty$ \\ \hline
$d_0^T$ & $d_0^\infty$ & $\frac{1}{|\zeta|}h_0$ \\ \hline
$d_0^\infty$ & $d_0^\infty$ & $d_0^\infty$ \\ \hline
$h_0$ & $h_0$ & $h_0$ \\ \hline
$\frac{1}{|\zeta|}h_0$ & $\frac{1}{|\zeta|}h_0$ & $\frac{1}{|\zeta|}h_0$ \\ \hline
$(1 + \frac{\theta}{|\zeta|})h_0$ & $\frac{1}{|\zeta|}h_0$ & $h_0$ \\ \hline
 \end{tabular}
 \end{center}
\end{table}

Obviously, there is no isometry between $(\R^3,h_0,0)$ 
and $(\R^3,\frac{1}{|\zeta|}h_0,0)$. 
In the next section, we show that $(\R^3,d_0^\infty,0)$ is isometric to 
neither $(\R^3,h_0,0)$ nor $(\R^3,\frac{1}{|\zeta|}h_0,0)$. 

Then, Table \ref{table1} implies that 
the nontrivial 
isometries may exist between 
\begin{align*}
(\R^3,d_S^\infty,0)\ &{\rm and}\ (\R^3,d_{S'}^\infty,0)\ {\rm for}\ S\neq S',\\ 
(\R^3,d_0^T,0)\ &{\rm and}\ (\R^3,d_0^{T'},0)\ {\rm for}\ T\neq T',\\
\Big(\R^3,\Big(1 + \frac{\theta}{|\zeta|}\Big)h_0,0\Big)\ &{\rm and}\ 
\Big(\R^3,\Big(1 + \frac{\theta'}{|\zeta|}\Big)h_0,0\Big)\ {\rm for}\ \theta\neq \theta'.
\end{align*}

Suppose $(\R^3,d_{S}^\infty,0)$ is isometric to $(\R^3,d_{S'}^\infty,0)$ 
for some $S\neq S'$. 
Then the topological space 
$$\{ (\R^3,d_{S}^\infty,0);\ S\in\R^+\}$$ 
with respect to pointed Gromov-Hausdorff topology 
is homeomorphic to $S^1$ or $1$-point, 
hence it is compact. 
Then its closure is itself, 
therefore $(\R^3,h_0,0)$ is isometric to some 
$(\R^3,d_{S}^\infty,0)$, 
which is the contradiction by Table \ref{table1}. 
Similarly, we can show that there are no isometries between 
$(\R^3,d_0^{T},0)$ and $(\R^3,d_0^{T'},0)$, 
and between $(\R^3,(1 + \frac{\theta}{|\zeta|})h_0,0)$ and 
$(\R^3,(1 + \frac{\theta'}{|\zeta|})h_0,0)$.

\subsection{Example (2)}\label{sec ex2}
Next we suppose that $\{ K_n\}_n$ 
satisfies 
\begin{align*}
\lim_{n\to \infty}\frac{K_{2n}}{K_{2n-1}} =\infty,\quad 
\frac{K_{2n+1}}{K_{2n}} =\beta>1. 
\end{align*}
Take $S>0$ and put 
$a_n:=K_{2n}^{-1-\alpha}S^{1+\alpha}$. 
Then we have $a_n^{\frac{1}{1+\alpha}}K_{2n} = S$ 
and $a_n^{\frac{1}{1+\alpha}}K_{2n+1} =\beta S$. 
Hence Theorem \ref{main1} implies that 
$(X,a_ng_\Lambda,p)\xrightarrow{GH} (\R^3,d_S^{\beta S})$. 
By arguing similarly to the proof of Theorem \ref{ex1} 
we obtain the followings. 

\begin{thm}\label{ex2}
Let $\Lambda,\{ K_n\}_n$ satisfy 
\begin{align*}
\lim_{n\to \infty}\frac{K_{2n}}{K_{2n-1}} =\infty,\quad 
\lim_{n\to \infty}\frac{K_{2n+1}}{K_{2n}} =\beta>1.
\end{align*}
Then $\mathcal{T}(X, g_\Lambda)$ is equal to 
the closure of 
\begin{align*}
\{ (\R^3, s d_1^\beta,0); s>0 \}\cup 
\{ (\R^3, s \Big( 1+\frac{1}{|\zeta|} \Big) h_0,0); s>0 \}
\end{align*}
with respect to the Gromov-Hausdorff topology. 
Moreover we have 
\begin{align*}
\lim_{s\to \infty} (\R^3, s d_1^\beta,0) 
&= \lim_{s\to 0} \Big(\R^3, s \Big( 1+\frac{1}{|\zeta|} \Big) h_0,0\Big)
= (\R^3,h_0,0),\\
\lim_{s\to 0}  (\R^3, s d_1^\beta,0)
&= \lim_{s\to \infty } \Big(\R^3, s \Big( 1+\frac{1}{|\zeta|} \Big),0\Big) 
= \Big(\R^3,\frac{1}{|\zeta|}h_0,0\Big). 
\end{align*}
\end{thm}

By the similar argument to Section \ref{sec ex1}, 
we can see that 
$(\R^3,d_S^{\beta S},0)$ is isometric to 
neither $(\R^3,h_0,0)$, $(\R^3,\frac{1}{|\zeta|}h_0,0)$ nor 
$(\R^3,d_{S'}^{\beta S'},0)$ for $S'\neq S$.

\subsection{Example (3)}\label{sec ex3}
For $I\subset \R^+$, denote by $d_I$ the metric on $\R^3$ induced by 
\begin{align*}
\int_{x\in I}\frac{dx}{|\zeta - (x^\alpha,0,0)|}\cdot h_0. 
\end{align*}
Denote by $\mathcal{B}_+(\R^+)$ the set consisting of 
all Borel subsets of $\R^+$ of nonzero Lebesgue measure. 
In this subsection we show the next theorem.

\begin{thm}\label{ex3}
There is a sequence $\{ K_n\}_n$ 
such that $\mathcal{T}(X,g_\Lambda)$ contains 
\begin{align*}
\{ (\R^3,d_I,0);\ I\in \mathcal{B}_+(\R^+)\} / \mbox{{\rm isometry}}.
\end{align*}
\end{thm}
\begin{proof}
Put 
\begin{align*}
\mathcal{O}_0 &:= \{ I\subset \R^+;\ I\ \mbox{is nonempty and open}\}, \\
\mathcal{O}_1 &:= 
\left\{
\begin{array}{l}
{}\\
{}
\end{array}
\right.
\bigcup_{i=1}^k (S_l,T_l) \subset \R^+; 
\left.
\begin{array}{l}
S_l,T_l\in\mathbb{Q},\ 1\le k<\infty,\\
0<S_l<T_l<S_{l+1}<\infty
\end{array}
\right\},
\end{align*}
then one can see 
$\mathcal{O}_1\subset \mathcal{O}_0\subset \mathcal{B}_+(\R^+)$. 
Since $\mathcal{O}_1$ is countable, we can label the open sets 
in $\mathcal{O}_1$ such as 
\begin{align*}
\mathcal{O}_1 = \{ I_1,I_2,I_3,\ldots\},\quad 
I_m = \bigcup_{l=1}^{k_m} (S_{m,l},T_{m,l}).
\end{align*}
Now, we fix a bijection $F:\N\to \N\times\N$ and write 
$F(q) = (i(q),m(q))$. 
Define $L_q>0$ inductively by 
\begin{align*}
L_{q+1} := 2^{i(q) + i(q+1)}L_q \cdot \frac{T_{m(q),k_{m(q)}}}{S_{m(q),1}},\quad 
L_0 := 1.
\end{align*}
Then we can define $0<K_0<K_1<\cdots$ such that 
\begin{align*}
\{ K_0<K_1<\cdots\} = \Big\{ L_q \frac{S_{m(q),l}}{S_{m(q),1}}, 
L_q \frac{T_{m(q),l}}{S_{m(q),1}};\ 1\le l\le k_{m(q)}, \ q=0,1,\ldots\Big\}.
\end{align*}

First of all we show $(\R^3, d_{I_m}, 0)\in \mathcal{T}(X,g_\Lambda)$ 
for every $I_m\in\mathcal{O}_1$. 
Fix $m$. 
For any $i\in\N$, we can take a unique $q$ such that 
$i(q)=i$ and $m(q)=m$. 
Put $a_i^\frac{1}{1+\alpha} := L_q^{-1}S_{m, 1}$, then we have 
\begin{align*}
a_i^\frac{1}{1+\alpha} L_q \frac{S_{m,l}}{S_{m,1}} = S_{m,l},\quad 
a_i^\frac{1}{1+\alpha} L_q \frac{T_{m,l}}{S_{m,1}} = T_{m,l}.
\end{align*}
Note that $L_{q+1}\ge 2^{i(q) + i(q+1)}L_q$ implies 
$L_q\to \infty$ as $i\to \infty$, 
hence $a_i\to 0$ as $i\to \infty$. 
Here, we put $\Phi=\Phi_{a_i}$ and 
$\Phi_\infty = \sum_{l=1}^{k_m}\Phi_{S_{m,l}}^{T_{m,l}}$. 
By applying Proposition \ref{conv1} and \eqref{est1}-\eqref{est5} 
with $P=1$,  the constants appearing in 
{\bf (A3-6)} are given by 
\begin{align*}
&\varepsilon = 2 a_i^\frac{1}{1+\alpha} + 2^{-i}S_{m,1} 
+\frac{2^{1-(\alpha-1)i}T_{m,k_m}^{-\alpha + 1}}{\alpha-1},\quad 
C_0 = \frac{1}{2}\sum_{l=1}^{k_m}A_{S_{m,l}}^{T_{m,l}},\\
&C_1 = \frac{\alpha 2^\frac{1}{\alpha}}{\alpha - 1},\quad 
m=1,\quad \kappa=\frac{1}{\alpha}, 
\end{align*}
if we suppose $\varepsilon$ is sufficiently small. 
One can see $\varepsilon \to 0$ as $i\to \infty$, then we obtain 
$\{ (X,a_i g_\Lambda,p)\}_i\xrightarrow{\text{GH}} (\R^3,d_{I_m},0)$.

Next we show that $(\R^3,d_I,0) \in \mathcal{T}(X,g_\Lambda)$ for
any $I\in \mathcal{O}_0$. 
To show it, we apply Vitali's Covering Theorem. 
Fix $I\in\mathcal{O}_0$ and put 
$\mathcal{I}:=\{ (a,b)\in\mathcal{O}_0;\ [a,b]\subset I\}$. 
Then $\mathcal{I}$ is a Vitali cover of $I$, hence there exists 
$\{ J_n\}_{n\in\N}\subset \mathcal{I}$ such that 
\begin{align*}
J_n\neq J_{n'}\ (\mbox{if}\ n\neq n'),\quad 
m\Big(I\backslash \bigsqcup_{n\in\N}J_n\Big) = 0,
\end{align*}
where $m$ is the Lebesgue measure. 
Put $\hat{J}_n:=\bigsqcup_{k=1}^nJ_k$. 
Since $\hat{J}_n\in\mathcal{O}_1$ holds, then 
$(\R^3,d_{\hat{J}_n},0)\in\mathcal{T}(X,g_\Lambda)$. 
If we put $\Phi_J(\zeta) := \int_{x\in J}\frac{dx}{|\zeta-(x^\alpha,0,0)|}$, 
then we can see 
\begin{align*}
|\Phi_{\hat{J}_n}(\zeta) - \Phi_I(\zeta)| 
\le \frac{ m(I\backslash \hat{J}_n) }{D} \to 0\quad (\mbox{as}\ n\to \infty), 
\end{align*}
and we can take the constants in {\bf (A3-6)} independent of $n$ 
by using Proposition \ref{lower1}.
Therefore, we obtain 
$\{ (\R^3,d_{\hat{J}_n},0)\}_n\xrightarrow{\text{GH}} (\R^3,d_I,0)$. 

Finally, let $I\in\mathcal{B}_+(\R^+)$. 
Since the Lebesgue measure is the Radon measure, 
there exist $U_n\subset \mathcal{O}_1$ for any $n$ such that 
$I\subset U$ and $m(U)\le m(I) + \frac{1}{n}$. 
Then we have $|\Phi_I(\zeta) - \Phi_{U_n}(\zeta)| \le \frac{1}{nD}$, 
we have $\{ (\R^3,d_{U_n},0)\}_n\xrightarrow{\text{GH}} (\R^3,d_I,0)$ 
by the similar argument. 
Here, the positivity of $m(I)$ is necessary since $C_0$ in 
{\bf (A5)} is given by $\int_{I}\frac{dx}{1+x^\alpha}>0$ by \eqref{est1}. 

\end{proof}

By Theorem \ref{ex3}, we can see 
$(\R^3,h_0,0)$ and $(\R^3,\frac{1}{|\zeta|}h_0,0)$ are 
also contained in $\mathcal{B}_+(\R^+)$. 
The author does not know whether any other metric spaces may 
appear as the tangent cone at infinity of $(X,g_\Lambda)$ or not.

\section{On the geometry of the limit spaces}\label{geometry}

In this section, we study the geometry of 
$(\R^3,d_0^\infty)$, 
and conclude that there are no isometry between 
$(\R^3,d_0^\infty)$ and $(\R^3,h_0)$, 
and between $(\R^3,d_0^\infty)$ and $(\R^3,\frac{1}{|\zeta|}h_0)$

\begin{prop}
$(\R^3,\frac{1}{|\zeta|}h_0)$ is the 
Riemannian cone $S^2\times \R^+$, 
where the Riemannian metric on $S^2$ is the homogeneous one whose 
area is equal to $\pi$.
\label{riemannian cone}
\end{prop}
\begin{proof}
Put $\zeta = (\zeta_1,\zeta_2,\zeta_3) \neq 0$ 
and $r=\sqrt{\zeta_1^2+\zeta_2^2+\zeta_3^2}$, 
and let $g_{S^2}$ be the standard Riemannian metric on $S^2$ with 
constant curvature and volume $4\pi$. 
Then by putting $R:=2\sqrt{r}$, we have 
\begin{align*}
\frac{1}{|\zeta|}h_0 = \frac{1}{r} ((dr)^2 + r^2 g_{S^2}) 
= (dR)^2 + R^2\cdot\frac{g_{S^2}}{4}. 
\end{align*}
\end{proof}

Next we review the notion of polar spaces, 
introduced by Cheeger and Colding in \cite{Cheeger-Colding1997}
then show that the metric space $(\R^3,d_0^\infty)$ 
never be a polar space.

Let $Y$ be a metric space, and suppose that 
there is a tangent cone $Y_y$ at $y\in Y$. 
Then we can consider tangent cones at any points in $Y_y$. 
The tangent cones obtained by repeating this process 
are called {\it iterated tangent cones} of $Y$.
A point $x$ in a length space $X$ is called a {\it pole} 
if there is a ray $\gamma:[0,\infty) \to X$ and $t \ge 0$ 
for any $\underline{x}\neq x$ such that $\gamma(0) = x$ 
and $\gamma(t) = \underline{x}$. 
Here, the ray $\gamma:[0,\infty) \to X$ is a continuous curve such that 
the length of $\gamma|_{[t_0,t_1]}$ is equal to $|\gamma(t_0)\gamma(t_1)|$.

\begin{definition}[\cite{Cheeger-Colding1997}]\normalfont
The metric space $Y$ is called a {\it polar space} 
if all of the base points of the iterated tangent cones of $Y$ are poles.
\end{definition}

For example, let $C(X)$ be a metric cone of a metric space $X$. 
Then every $\gamma$ defined by $\gamma(t):=(x,t) \in X\times \R^+ =C(X)$ 
is a ray, hence the base points of any metric cones are poles. 
Now, since $(\R^3,\frac{1}{|\zeta|}h_0)$ is a Riemannian cone of a 
smooth compact Riemannian manifold, 
then all of the iterated tangent cones are $(\R^3,\frac{1}{|\zeta|}h_0)$ itself 
or $(\R^3,h_0)$. 
Consequently, we can conclude that $(\R^3,\frac{1}{|\zeta|}h_0)$ is polar.
Obviously, $(\R^3,h_0)$ is also polar. 
We can also see in the similar way 
that $(\R^3,(1+\frac{\theta}{|\zeta|})h_0)$ is polar. 
On the other hand we can show the next proposition. 
\begin{prop}
The origin $0\in\R^3$ is not a pole 
of the metric space $(\R^3,d_0^\infty)$. 
In particular, $(\R^3,d_0^\infty)$ is neither a polar space 
nor a metric cone of any metric spaces.
\label{nonpolar}
\end{prop}
\begin{proof}
First of all we show that $0\in\R^3$ is not a pole with respect to $d_0^\infty$. 
Put $p:=(1,0,0) \in \R^3$, and suppose that there is a ray 
$\gamma:[0,\infty) \to \R^3$ such that $\gamma(0) = 0$ and $\gamma(t_0)=p$ 
for some $t_0 >0$. 
Then we have 
\begin{align*}
d_0^\infty (\gamma(s_0),\gamma(s_1)) 
= \int_{s_0}^{s_1}\sqrt{\Phi_0^\infty(\gamma(t))}|\gamma'(t)|dt
\end{align*}
for any $0\le s_0<s_1$.

For $\delta >0$, let 
\begin{align*}
A_\delta := \{ t\in \R; |\gamma_\C(t)| \ge \delta \}.
\end{align*}
Then there is a sufficiently small $\delta$ such that 
$A_\delta \cap (0,t_0) \neq \emptyset$ and $A_\delta \cap (t_0,\infty) \neq \emptyset$.
This is because the length of $\gamma|_I$ becomes infinity 
for any small interval $I\subset \R$ if not. 
Since $A_\delta$ is closed and does not contain $t_0$, we can take 
a connected component $(a_0,a_1)$ of 
$\R\backslash A_\delta$ containing $t_0$.
Then we can see that $|\gamma_\C(a_0)|=|\gamma_\C(a_1)|=\delta$ and 
$|\gamma_\C(t)| < \delta$ for any $t\in (a_0,a_1)$. 
Now define $\tilde{\gamma}:[0,a_1]\to X$ by 
\[ \tilde{\gamma}(t) := 
\left \{
\begin{array}{cc}
(\gamma_\R(t),e^{i\theta}\gamma_\C(t)) & (0\le t \le a_0) \\
e^{i\theta}P_{\gamma|_{[a_0, a_1]}}(t) & (a_0\le t \le a_1) 
\end{array}
\right.
\]
where $\theta$ is defined by $e^{i\theta}\gamma_\C(a_0) = \gamma_\C(a_1)$. 
Recall that $P_{\gamma|_{[a_0, a_1]}}$ is already defined 
in Lemma \ref{lowreplace}. 
Then by applying Lemma \ref{lowreplace}, 
we can see that the length of $\tilde{\gamma}$ is strictly less than 
the length of $\gamma|_{[0,a_1]}$, 
therefore $\gamma$ is not the ray, which is the contradiction. 
Hence $0\in\R^3$ is not the pole.

Now we can check that the $\R^+$-action on $\R^3$ defined by 
the scalar multiplication 
is homothetic with respect to $d_0^\infty$, 
then the tangent cone of $(\R^3,d_0^\infty)$ 
at $0$ is itself. 
Consequently, $(\R^3,d_0^\infty)$ is not a polar space.

Suppose that $(\R^3,d_0^\infty)$ is the metric cone 
of some metric spaces $X$, 
then the origin $0$ is nothing but the base point of the metric cone.
Since the base point of the metric cone is always a pole, 
hence we have the contradiction. 
\end{proof}
Now we obtain the next corollary.
\begin{cor}
There is no isometry between 
$(\R^3,d_0^\infty)$ and $(\R^3,h_0)$, and 
between $(\R^3,d_0^\infty)$ and $(\R^3,\frac{1}{|\zeta|}h_0)$.
\end{cor}

\bibliography{asymptotic}

\begin{thebibliography}{10}

\bibitem{AKL1989}
Michael~T. Anderson, Peter~B. Kronheimer, and Claude LeBrun.
\newblock Complete {R}icci-flat {K}\"ahler manifolds of infinite topological
  type.
\newblock {\em Comm. Math. Phys.}, 125(4):637--642, 1989.

\bibitem{berestovskii1987submetry}
V.~N. Berestovski{\u\i}.
\newblock ``{S}ubmetries'' of three-dimensional forms of nonnegative curvature.
\newblock {\em Sibirsk. Mat. Zh.}, 28(4):44--56, 224, 1987.

\bibitem{berestovskii2000submetry}
V.~N. Berestovskii and Luis Guijarro.
\newblock A metric characterization of {R}iemannian submersions.
\newblock {\em Ann. Global Anal. Geom.}, 18(6):577--588, 2000.

\bibitem{Cheeger-Colding1996}
Jeff Cheeger and Tobias~H. Colding.
\newblock Lower bounds on {R}icci curvature and the almost rigidity of warped
  products.
\newblock {\em Ann. of Math. (2)}, 144(1):189--237, 1996.

\bibitem{Cheeger-Colding1997}
Jeff Cheeger and Tobias~H. Colding.
\newblock On the structure of spaces with {R}icci curvature bounded below. {I}.
\newblock {\em J. Differential Geom.}, 46(3):406--480, 1997.

\bibitem{Colding-Minicozzi2014}
Tobias~Holck Colding and William~P. Minicozzi, II.
\newblock On uniqueness of tangent cones for {E}instein manifolds.
\newblock {\em Invent. Math.}, 196(3):515--588, 2014.

\bibitem{Colding-Naber2013}
Tobias~Holck Colding and Aaron Naber.
\newblock Characterization of tangent cones of noncollapsed limits with lower
  {R}icci bounds and applications.
\newblock {\em Geom. Funct. Anal.}, 23(1):134--148, 2013.

\bibitem{goto1994}
R.~Goto.
\newblock On hyper-{K}\"ahler manifolds of type {$A_\infty$}.
\newblock {\em Geom. Funct. Anal.}, 4(4):424--454, 1994.

\bibitem{gromov1981}
Mikhael Gromov.
\newblock {\em Structures m\'etriques pour les vari\'et\'es riemanniennes},
  volume~1 of {\em Textes Math\'ematiques [Mathematical Texts]}.
\newblock CEDIC, Paris, 1981.
\newblock Edited by J. Lafontaine and P. Pansu.

\bibitem{gromov2007}
Misha Gromov.
\newblock {\em Metric structures for {R}iemannian and non-{R}iemannian spaces}.
\newblock Modern Birkh\"auser Classics. Birkh\"auser Boston, Inc., Boston, MA,
  english edition, 2007.
\newblock Based on the 1981 French original, With appendices by M. Katz, P.
  Pansu and S. Semmes, Translated from the French by Sean Michael Bates.

\bibitem{hattori2011}
Kota Hattori.
\newblock The volume growth of hyper-{K}\"ahler manifolds of type {$A_\infty$}.
\newblock {\em J. Geom. Anal.}, 21(4):920--949, 2011.

\bibitem{perelman1997}
G.~Perelman.
\newblock A complete {R}iemannian manifold of positive {R}icci curvature with
  {E}uclidean volume growth and nonunique asymptotic cone.
\newblock In {\em Comparison geometry ({B}erkeley, {CA}, 1993--94)}, volume~30
  of {\em Math. Sci. Res. Inst. Publ.}, pages 165--166. Cambridge Univ. Press,
  Cambridge, 1997.

\end{thebibliography}
\end{document}